\documentclass[12pt]{amsart}
\usepackage{amscd,amsmath,amssymb,enumerate,amsfonts,mathrsfs,txfonts}
\usepackage{comment}
\usepackage{hyperref}
\usepackage{xcolor}
\usepackage{tikz,pgfplots,pgf}
\usepackage{vmargin}
\usepackage[all]{xy}

\tikzset{node distance=3cm, auto}
\usetikzlibrary{calc} 
\usetikzlibrary{trees} 
\setpapersize{A4}
\setmargins{2.5cm}      
{1.5cm}                        
{16.5cm}                      
{23.42cm}                   
{10pt}                          
{1cm}                           
{0pt}                          
{2cm}
\newtheorem{theorem}{Theorem}[section]
\newtheorem{lemma}[theorem]{Lemma}
\newtheorem{proposition}[theorem]{Proposition}
\newtheorem{corollary}[theorem]{Corollary}
\theoremstyle{definition}
\newtheorem{definition}[theorem]{Definition}

\newtheorem{remark}[theorem]{Remark}
\numberwithin{equation}{section}

\newtheorem{cd}{Condition}

\def\N{\mathbb{N}}

\def\C{\mathbb{C}}
\def\K{\mathbb{K}}

\def\G{\mathcal{G}}
\def\H{\mathcal{H}}
\def\I{\mathcal{I}}
\def\L{\mathcal{L}}
\def\D{\mathcal{D}}
\def\lin{\mathrm{lin}}

\def\rank{\mathrm{rank}}

\def\id{\mathrm{id}}

\makeatletter
\@namedef{subjclassname@2020}{\textup{2020} Mathematics Subject Classification}
\makeatother

\begin{document}
\title[Cohen strongly $p$-summing holomorphic mappings on Banach spaces]{Cohen strongly $p$-summing holomorphic mappings \\ on Banach spaces}

\author[A. Jim{\'e}nez-Vargas]{A. Jim\'enez-Vargas}
\address[A. Jim{\'e}nez-Vargas]{Departamento de Matem\'aticas, Universidad de Almer\'ia, 04120, Almer\'ia, Spain}
\email{ajimenez@ual.es}

\author[K. Saadi]{K. Saadi}
\address[K. Saadi]{Laboratoire d'Analyse Fonctionnelle et G\'eom\'etrie des Espaces, University of M'sila, 28000 M'sila, Algeria}
\email{kh\_saadi@yahoo.fr}

\author[J. M. Sepulcre]{J. M. Sepulcre}
\address[J. M. Sepulcre]{Departamento de Matem\'aticas, Universidad de Alicante, 03080 Alicante, Spain}
\email{JM.Sepulcre@ua.es}

\thanks{Research of the first author was partially supported by project UAL-FEDER grant UAL2020-FQM-B1858, by Junta de Andaluc\'{\i}a grants P20$\_$00255 and FQM194, and by Ministerio de Ciencia e Innovaci\'on grant PID2021-122126NB-C31. The third author was also supported by PGC2018-097960-B-C22 (MCIU/AEI/ERDF, UE)}
\date{\today}

\begin{abstract}
Let $E$ and $F$ be complex Banach spaces, $U$ be an open subset of $E$ and $1\leq p\leq\infty$. We introduce and study the notion of a Cohen strongly $p$-summing holomorphic mapping from $U$ to $F$, a holomorphic version of a strongly $p$-summing linear operator. For such mappings, we establish both Pietsch domination/factorization theorems and analyse their linearizations from $\G^\infty(U)$ (the canonical predual of $\H^\infty(U)$) and their transpositions on $\H^\infty(U)$. Concerning the space $\D_p^{\H^\infty}(U,F)$ formed by such mappings and endowed with a natural norm $d_p^{\H^\infty}$, we show that it is a regular Banach ideal of bounded holomorphic mappings generated by composition with the ideal of strongly $p$-summing linear operators. Moreover, we identify the space $(\D_p^{\H^\infty}(U,F^*),d_p^{\H^\infty})$ with the dual of the completion of tensor product space $\G^\infty(U)\otimes F$ endowed with the Chevet--Saphar norm $g_p$.
\end{abstract}
\maketitle


\section*{Introduction}

The linear theory of absolutely summing operators between Banach spaces was initiated by Grothendieck \cite{Gro-55} in 1950 with the introduction of the concept of $1$-summing operator. In 1967, Pietsch \cite{Pie-67} defined the class of absolutely $p$-summing operators for any $p>0$ and established many of their fundamental properties.

The nonlinear theory for such operators started with Pietsch \cite{Pie-84} in 1983. Since then, the idea of extending the theory of absolutely $p$-summing operators to other settings has been developed by various authors, namely, the polynomial, multilinear, Lipschitz and holomorphic settings (see, for example, \cite{AchMez-07,AngFer-20,Dim-03,FarJoh-09,Pel-03,Saa-15,YahAchRue-16}). In 1996, Matos \cite{Mat-96} obtained the first results about absolutely summing holomorphic mappings between Banach spaces. 
Our approach in this paper is different from that of Matos. 

In 1973, Cohen \cite{Coh-73} introduced the concept of a strongly $p$-summing linear operator to characterize those operators whose adjoints are absolutely $p^*$-summing operators, where $p^*$ denotes the conjugate index of $p\in (1,\infty]$. Influenced by this class of operators, we introduce and study a new concept of summability in the category of bounded holomorphic mappings, which yields the called \textit{Cohen strongly $p$-summing holomorphic mappings}. 

We now describe the contents of the paper. Let $E$ and $F$ be complex Banach spaces, $U$ be an open subset of $E$ and $1\leq p\leq\infty$. We denote by $\H^\infty(U,F)$ the Banach space of all bounded holomorphic mappings from $U$ to $F$, equipped with the supremum norm. In particular, $\H^\infty(U)$ stands for the space $\H^\infty(U,\C)$. It is known that $\H^\infty(U)$ is a dual Banach space whose canonical predual, denoted $\G^\infty(U)$, is the norm-closed linear subspace of $\H^\infty(U)^*$ generated by the evaluation functionals at the points of $U$.

In Section \ref{1}, we fix the notations and recall some results on the space $\H^\infty(U,F)$, essentially, a remarkable linearization theorem due to Mujica \cite{Muj-91} which is a key tool to establish our results. 

In Section \ref{2}, we show that the space of all Cohen strongly $p$-summing holomorphic mappings from $U$ to $F$, denoted $\D_p^{\H^\infty}(U,F)$ and equipped with a natural norm $d^{\H^\infty}_p$, is a regular Banach ideal of bounded holomorphic mappings. Furthermore, $\mathcal{D}^{\H^\infty}_1(U,F)=\H^\infty(U,F)$.

It is well known that the elements of the tensor product of two linear spaces can be viewed as linear mappings or bilinear forms. Following this idea, in Section \ref{3} we introduce the tensor product $\Delta(U)\otimes F$ as a space of linear functionals on the space $\H^\infty(U,F^*)$, and equip this space with the known Chevet--Saphar norms $g_p$ and $d_p$. 

Section \ref{4} addresses the duality theory: the space $(\D_p^{\H^\infty}(U,F^*),d^{\H^\infty}_p)$ is canonically isometrically isomorphic to the dual of the completion of the tensor product space $\G^\infty(U)\otimes_{g_p} F$. In particular, we deduce that $\H^\infty(U,F^*)$ is a dual space.

Pietsch \cite{Pie-67} established a domination/factorization theorem for $p$-summing linear operators between Banach spaces. Characterizing previously the elements of the dual space of $\Delta(U)\otimes_{g_p} F$, we present for Cohen strongly $p$-summing holomorphic mappings both versions of Pietsch domination theorem and Pietsch factorization theorem in Sections \ref{5} and \ref{6}, respectively. 

Moreover, in Section \ref{5}, we prove that a mapping $f\colon U\to F$ is Cohen strongly $p$-summing holomorphic if and only if Mujica's linearization $T_f\colon\G^\infty(U)\to F$ is a strongly $p$-summing operator. Several interesting applications of this fact are obtained. 

On the one hand, we show that the ideal $\D_p^{\H^\infty}(U,F)$ is generated by composition with the ideal $\D_p$ of strongly $p$-summing linear operators, that is, every mapping $f\in\D_p^{\H^\infty}(U,F)$ admits a factorization in the form $f=T\circ g$, for some complex Banach space $G$, $g\in\H^{\infty}(U,G)$ and $T\in\D_{p}(G,F)$. Moreover, $d_p^{\H^\infty}(f)$ coincides with $\inf\{d_p(T)\left\|g\right\|_\infty\}$, where the infimum is extended over all such factorizations of $f$, and, curiously, this infimum is attained at Mujica's factorization of $f$. We also show that every $f\in\D_2^{\H^\infty}(U,F)$ factors through a Hilbert space whenever $F$ is reflexive, and establish some inclusion and coincidence properties of spaces $\D_p^{\H^\infty}(U,F)$. 

On the other hand, we analyse holomorphic transposition of their elements and prove that every member of $\D_p^{\H^\infty}(U,F)$ has relatively weakly compact range that becomes relatively compact whenever $F$ is reflexive. Let us recall that the study of holomorphic mappings with relatively (weakly) compact range was initiated by Mujica \cite{Muj-91} and followed in \cite{JimRuiSep-22}. 


\section{Notations and preliminaries}\label{1}

Throughout this paper, unless otherwise stated, $E$ and $F$ will denote complex Banach spaces and $U$ an open subset of $E$.

We first introduce some notations. As usual, $B_E$ denotes the closed unit ball of $E$. For two vector spaces $E$ and $F$, $L(E,F)$ stands for the vector space of all linear operators from $E$ into $F$. In the case that $E$ and $F$ are normed spaces, $\mathcal{L}(E,F)$ represents the normed space of all bounded linear operators from $E$ to $F$ endowed with the canonical norm of operators. In particular, the algebraic dual $L(E,\mathbb{K})$ and the topological dual $\mathcal{L}(E,\mathbb{K})$ are denoted by $E^{\prime }$ and $E^*$, respectively. For each $e\in E$ and $e^*\in E^{\prime }$, we frequently will write $\langle e^*,e\rangle$ instead of $e^*(e)$. We denote by $\kappa_E$ the canonical isometric embedding of $E$ into $E^{**}$ defined by $\left\langle \kappa_E(e),e^*\right\rangle=\left\langle e^*,e\right\rangle$ for $e\in E$ and $e^*\in E^*$. For a set $A\subseteq E$, $\mathrm{co}(A)$ denotes the convex hull of $A$.

We now recall some concepts and results of the theory of holomorphic mappings on Banach spaces.

\begin{theorem}
\label{teo-0}(see \cite[7 Theorem]{Muj-05} and \cite[Theorem 8.7]{Muj-86})
Let $E$ and $F$ be complex Banach spaces and let $U$ be an open set in $E$. For a mapping $f\colon U\to F$, the following conditions are equivalent:
\begin{enumerate}
\item For each $a\in U$, there is an operator $T\in\mathcal{L}(E,F)$ such that 
$$
\lim_{x\to a}\frac{f(x)-f(a)-T(x-a)}{\left\|x-a\right\|}=0. 
$$
\item For each $a\in U$, there exist an open ball $B(a,r)\subseteq U$ and a sequence of continuous $m$-homogeneous polynomials $(P_{m,a})_{m\in\mathbb{N}_0}$ from $E$ into $F$ such that 
$$f(x)=\sum_{m=0}^\infty P_{m,a}(x-a), $$where the series converges uniformly for $x\in B(a,r)$.
\item $f$ is G-holomorphic (that is, for all $a\in U$ and $b\in E$, the mapping $\lambda\mapsto f(a+\lambda b)$ is holomorphic on the open set $\{\lambda\in\mathbb{C}\colon a+\lambda b\in U\}$) and continuous. $\hfill\qed$
\end{enumerate}
\end{theorem}

A mapping $f\colon U\to F$ is said to be \textit{holomorphic} if it verifies the equivalent conditions in Theorem \ref{teo-0}. The mapping $T$ in condition (i) is uniquely determined by $f$ and $a$, and is called the \textit{differential of $f$ at $a$} and denoted by $Df(a)$.


A mapping $f\colon U\to F$ is \textit{locally bounded} if $f$ is bounded on a suitable neighborhood of each point of $U$. Given a Banach space $E$, a subset $N\subseteq B_{E^*}$ is said to be \textit{norming for $E$} if the functional 
$$
N(x)=\sup\left\{\left|x^*(x)\right|\colon x^*\in N\right\}\qquad (x\in E) 
$$
defines the norm on $E$.

If $U\subseteq E$ and $V\subseteq F$ are open sets, $\H(U,V)$ will represent the set of all holomorphic mappings from $U$ to $V$. We will denote by $\H(U,F)$ the linear space of all holomorphic mappings from $U$ into $F$ and by $\H^\infty(U,F)$ the subspace of all $f\in\H(U,F)$ such that $f(U)$ is bounded in $F$. When $F=\mathbb{C}$, then we will write $\H^\infty(U,\mathbb{C})=\H^\infty(U)$.

It is easy to prove that the linear space $\H^\infty(U,F)$, equipped with the supremum norm: 
$$
\left\|f\right\|_\infty=\sup\left\{\left\|f(x)\right\|\colon x\in U\right\}\qquad \left(f\in\H^\infty(U)\right), 
$$
is a Banach space. Let $\mathcal{G}^\infty(U)$ denote the norm-closed linear hull in $\H^\infty(U)^*$ of the set $\left\{\delta(x)\colon x\in U\right\}$ of \textit{evaluation functionals} defined by 
$$
\left\langle \delta(x),f\right\rangle=f(x)\qquad \left(f\in\H^\infty(U)\right). 
$$
In \cite{Muj-91,Muj-92}, Mujica established the following properties of $\mathcal{G}^\infty(U)$.

\begin{theorem}\label{teo1}\cite[Theorem 2.1]{Muj-91} 
Let $E$ be a complex Banach space and let $U$ be an open set in $E$.
\begin{enumerate}
\item $\H^\infty(U)$ is isometrically isomorphic to $\G^\infty(U)^*$, via the evaluation mapping $J_U\colon\H^\infty(U)\to\mathcal{G}^\infty(U)^*$ given by 
$$
\left\langle J_U(f),\gamma\right\rangle=\gamma(f)\qquad \left(\gamma\in\mathcal{G}^\infty(U),\; f\in\H^\infty(U)\right). 
$$
\item The mapping $g_U\colon U\to\G^\infty(U)$ defined by $g_U(x)=\delta(x)$ is holomorphic with $\left\|g_U(x)\right\|=1$ for all $x\in U$. 
\item For each complex Banach space $F$ and each mapping $f\in\H^\infty(U,F)$, there exists a unique operator $T_f\in\mathcal{L}(\mathcal{G}^\infty(U),F)$ such that $T_f\circ g_U=f$. 
Furthermore, $\left\|T_f\right\|=\left\|f\right\|_\infty$.
\item The mapping $f\mapsto T_f$ is an isometric isomorphism from $\H^\infty(U,F)$ onto $\mathcal{L}(\mathcal{G}^\infty(U),F)$.
\item \cite[Corollary 4.12]{Muj-91} (see also \cite[Theorem 5.1]{Muj-92}). $\G^\infty(U)$ consists of all functionals $\gamma\in\H^\infty(U)^*$ of the form
$$
\gamma=\sum_{i=1}^{\infty}\lambda_i\delta(x_i)
$$
with $(\lambda_i)_{i\geq 1}\in\ell_1$ and $(x_i)_{i\geq 1}\in U^\N$. Moreover,
$$
\left\|\gamma\right\|=\inf\left\{\sum_{i=1}^{\infty}\left|\lambda_i\right|\right\}
$$
where the infimum is taken over all such representations of $\gamma$.$\hfill\qed$
\end{enumerate}
\end{theorem}


\section{Cohen strongly $p$-summing holomorphic mappings}\label{2}

Let $E$ and $F$ be Banach spaces and $1\leq p\leq\infty$. Let us recall \cite{DisJarTon-95} that an operator $T\in\mathcal{L}(E,F)$ is \textit{$p$-summing} if there exists a constant $C\geq 0$ such that, regardless of the natural number $n$ and regardless of the choice of vectors $x_1,\ldots,x_n$ in $E$, we have the inequalities: 
\begin{align*}
\left(\sum_{i=1}^n\left\|T(x_i)\right\|^p\right)^{\frac{1}{p}}&\leq C \sup_{x^*\in B_{E^*}}\left(\sum_{i=1}^n\left|x^*(x_i)\right|^p\right)^{\frac{1}{p}}\quad & \text{if}\quad &1\leq p<\infty, \\
\max_{1\leq i\leq n}\left\|T(x_i)\right\|&\leq C \sup_{x^*\in B_{E^*}}\left(\max_{1\leq i\leq n}\left|x^*(x_i)\right|\right) \quad & \text{if}\quad &p=\infty.
\end{align*}
The infimum of such constants $C$ is denoted by $\pi_p(T)$ and the linear space of all $p$-summing operators from $E$ into $F$ by $\Pi_p(E,F)$.

The analogous notion for holomorphic mappings could be introduced as follows.

\begin{definition}
Let $E$ and $F$ be complex Banach spaces, let $U$ be an open subset of $E$, and let $1\leq p\leq\infty$. A holomorphic mapping $f\colon U\to F$ is said to be \textit{$p$-summing} if there exists a constant $C\geq 0$ such that for all $n\in\mathbb{N}$ and $x_1,\ldots,x_n\in U$, we have 
\begin{align*}
\left(\sum_{i=1}^n\left\|f(x_i)\right\|^p\right)^{\frac{1}{p}}&\leq C
\sup_{g\in B_{\H^\infty(U)}}\left(\sum_{i=1}^n\left|g(x_i)\right|^p\right)^{\frac{1}{p}}\quad & \text{if}\quad& 1\leq p<\infty, \\
\max_{1\leq i\leq n}\left\|f(x_i)\right\|&\leq C \sup_{g\in B_{\H^\infty(U)}}\left(\max_{1\leq i\leq n}\left|g(x_i)\right|\right) \quad & \text{if}\quad &p=\infty.
\end{align*}
We denote by $\pi^{\H^\infty}_p(f)$ the infimum of such constants $C$, and by $\Pi^{\H^\infty}_p(U,F)$ the set of all $p$-summing holomorphic mappings from $U$ into $F$.
\end{definition}

Notice that $p$-summing holomorphic mappings are not worth attention since 
$$
\Pi^{\H^\infty}_p(U,F)=\H^\infty(U,F) 
$$
with $\pi^{\H^\infty}_p(f)=\left\|f\right\|_\infty$ for all $f\in\Pi^{\H^\infty}_p(U,F)$. 

Let $1\leq p\leq\infty$ and let $p^*$ denote the \textit{conjugate index of $p$} given by 
$$
p^*=\left\{
\begin{array}{cll}
\displaystyle\frac{p}{p-1} & \text{if} & 1<p<\infty, \\ 
\infty & \text{if} & p=1, \\ 
1 & \text{if} & p=\infty .
\end{array}%
\right. 
$$
In \cite{Coh-73}, Cohen introduced the following subclass of $p$-summing operators between Banach spaces: an operator $T\in\mathcal{L}(E,F)$ is \textit{strongly $p$-summing} if there exists a constant $C\geq 0$ such that for all $n\in\mathbb{N}$, $x_1,\ldots,x_n\in E$ and $y^*_1,\ldots,y^*_n\in F^*$, we have 
\begin{align*}
\sum_{i=1}^n\left|\left\langle y^*_i,T(x_i)\right\rangle\right|&\leq C\left(\sum_{i=1}^n\left\|x_i\right\|\right) \sup_{y^{**}\in B_{F^{**}}}\left(\max_{1\leq i\leq n}\left|y^{**}(y^*_i)\right|\right)\quad
& \text{if}\quad &p=1, \\
\sum_{i=1}^n\left|\left\langle y^*_i,T(x_i)\right\rangle\right|&\leq C\left(\sum_{i=1}^n\left\|x_i\right\|^p\right)^{\frac{1}{p}} \sup_{y^{**}\in B_{F^{**}}}\left(\sum_{i=1}^n\left|y^{**}(y^*_i)\right|^{p^*}\right)^{\frac{1}{p^*}}\quad & \text{if}\quad &1<p<\infty, \\
\sum_{i=1}^n\left|\left\langle y^*_i,T(x_i)\right\rangle\right|&\leq C\left(\max_{1\leq i\leq n}\left\|x_i\right\|\right) \sup_{y^{**}\in B_{F^{**}}}\left(\sum_{i=1}^n\left|y^{**}(y^*_i)\right|\right)\quad & \text{if}\quad &p=\infty.
\end{align*}
The infimum of such constants $C$ is denoted by $d_p(T)$, and the space of all strongly $p$-summing operators from $E$ into $F$ by $\mathcal{D}_p(E,F)$. If $p=1$, we have $\mathcal{D}_1(E,F)=\mathcal{L}(E,F)$.

We now introduce a version of this concept in the setting of holomorphic mappings.

\begin{definition}\label{def-Csps}
Let $E$ and $F$ be complex Banach spaces, let $U$ be an open subset of $E$, and let $1\leq p\leq\infty$. A holomorphic mapping $f\colon U\to F$ is said to be \textit{Cohen strongly $p$-summing} if there exists a constant $C\geq 0$ such that for all $n\in\mathbb{N}$, $\lambda_1,\ldots,\lambda_n\in\mathbb{C}$, $x_1,\ldots,x_n\in U$ and $y^*_1,\ldots,y^*_n\in F^*$, we have 
\begin{align*}
\sum_{i=1}^n\left|\lambda_i\right|\left|\left\langle y^*_i,f(x_i)\right\rangle\right|&\leq C\left(\sum_{i=1}^n\left|\lambda_i\right|\right) \sup_{y^{**}\in B_{F^{**}}}\left(\max_{1\leq i\leq n}\left|y^{**}(y^*_i)\right|\right)\quad & \text{if}\quad &p=1, \\
\sum_{i=1}^n\left|\lambda_i\right|\left|\left\langle y^*_i,f(x_i)\right\rangle\right|&\leq C\left(\sum_{i=1}^n\left|\lambda_i\right|^p\right)^{\frac{1}{p}} \sup_{y^{**}\in B_{F^{**}}}\left(\sum_{i=1}^n\left|y^{**}(y^*_i)\right|^{p^*}\right)^{\frac{1}{p^*}}\quad & \text{if}\quad &1<p<\infty, \\
\sum_{i=1}^n\left|\lambda_i\right|\left|\left\langle y^*_i,f(x_i)\right\rangle\right|&\leq C\left(\max_{1\leq i\leq n}\left|\lambda_i\right|\right) \sup_{y^{**}\in B_{F^{**}}}\left(\sum_{i=1}^n\left|y^{**}(y^*_i)\right|\right)\quad & \text{if}\quad &p=\infty.
\end{align*}
We denote by $d^{\H^\infty}_p(f)$ the infimum of such constants $C$, and by $\mathcal{D}_p^{\H^\infty}(U,F)$ the set of all Cohen strongly $p$-summing holomorphic mappings from $U$ into $F$. 
\end{definition}

We will show that $\D^{\H^\infty}_1(U,F)=\H^\infty(U,F)$ (see Proposition \ref{ideal summing}).

The concept of an ideal of bounded holomorphic mappings is inspired by the analogous one for bounded linear operators between Banach spaces \cite[Section 8.2]{Rya-02}.

\begin{definition}
\label{def-ideal} An \textit{ideal of bounded holomorphic mappings} (or simply, a \textit{bounded-holomorphic ideal}) is a subclass $\I^{\H^\infty}$ of $\H^\infty$ such that for each complex
Banach space $E$, each open subset $U$ of $E$ and each complex Banach space $F$, the components 
$$
\I^{\H^\infty}(U,F):=\I^{\H^\infty}\cap\H^\infty(U,F) 
$$
satisfy the following properties:

\begin{enumerate}
\item[(I1)] $\I^{\H^\infty}(U,F)$ is a linear subspace of $\H^\infty(U,F)$,
\item[(I2)] For any $g\in\H^\infty(U)$ and $y\in F$, the mapping $g\cdot y\colon x\mapsto g(x)y$ from $U$ to $F$ is in $\I^{\mathcal{H}^\infty}(U,F)$,
\item[(I3)] \textit{The ideal property}: If $H,G$ are complex Banach spaces, $V$ is an open subset of $H$, $h\in\H(V,U)$, $f\in\I^{\H^\infty}(U,F)$ and $S\in\L(F,G)$, then $S\circ f\circ h$ is in $\I^{\H^\infty}(V,G)$.
\end{enumerate}

A bounded-holomorphic ideal $\I^{\H^\infty}$ is said to be \textit{normed (Banach)} if there exists a function $\left\|\cdot\right\|_{\I^{\H^\infty}}\colon\I^{\H^\infty}\to\mathbb{R}_0^+$ such that for every complex Banach space $E$, every open subset $U$ of $E$ and every complex Banach space $F$, the following conditions are satisfied:

\begin{enumerate}
\item[(N1)] $(\I^{\H^\infty}(U,F),\left\|\cdot\right\|_{\I^{\H^\infty}})$ is a normed (Banach) space with $\left\|f\right\|_\infty\leq\left\|f\right\|_{\I^{\H^\infty}}$ for all $f\in\I^{\H^\infty}(U,F)$, 
\item[(N2)] $\left\|g\cdot y\right\|_{\I^{\H^\infty}}=\left\|g\right\|_\infty\left\|y\right\|$ for every $g\in\H^\infty(U)$ and $y\in F$, 
\item[(N3)] If $H,G$ are complex Banach spaces, $V$ is an open subset of $H$, $h\in\H(V,U)$, $f\in\I^{\H^\infty}(U,F)$ and $S\in\L(F,G)$, then $\left\|S\circ f\circ h\right\|_{\I^{\H^\infty}}\leq \left\|S\right\|\left\|f\right\|_{\I^{\H^\infty}}$.
\end{enumerate}

A normed bounded-holomorphic ideal $\I^{\H^\infty}$ is said to be \textit{regular} if for any $f\in\H^\infty(U,F)$, we have that $f\in\I^{\H^\infty}(U,F)$ with $\left\|f\right\|_{\I^{\H^\infty}}=\left\|\kappa_F\circ f\right\|_{\I^{\H^\infty}}$ whenever $\kappa_F\circ f\in\I^{\H^\infty}(U,F^{**})$. 
\end{definition}

The following class of bounded holomorphic mappings appears involved in Definition \ref{def-ideal}.

\begin{lemma}
\label{lem-3} 
Let $g\in\H^\infty(U)$ and $y\in F$. The mapping $g\cdot y\colon U\to F$, given by $(g\cdot y)(x)=g(x)y$, belongs to $\H^\infty(U,F)$ with $\left\|g\cdot y\right\|_\infty=\left\|g\right\|_\infty\left\|y\right\|$.
\end{lemma}

\begin{proof}
It is clear that $\left\|(g\cdot y)(x)\right\|=|g(x)|\left\|y\right\|\leq\left\|g\right\|_\infty\left\|y\right\|$ for all $x\in U$, and thus $g\cdot y$ is bounded with $\left\|g\cdot y\right\|_\infty\leq\left\|g\right\|_\infty\left\|y\right\|$. For the converse inequality, note that $\left|g(x)\right|\left\|y\right\|=\left\|(g\cdot y)(x)\right\|\leq \left\|g\cdot y\right\|_\infty$ for all $x\in U$, and thus $\left\|g\right\|_\infty\left\|y\right\|\leq\left\|g\cdot y\right\|_\infty$.

We now prove that $g\cdot y$ is holomorphic. Given $a\in U\subseteq E$, since $g\in\H^\infty(U)$ there exists an unique functional $Dg(a)\in E^*$ such that 
$$
\lim_{x\to a}\frac{g(x)-g(a)-Dg(a)(x-a)}{\left\|x-a\right\|}=0. 
$$
Clearly, the mapping $T(a)\colon x\mapsto yDg(a)(x)$ from $E$ to $F$ is in $\mathcal{L}(E,F)$ and since 
$$
(g\cdot y)(x)-(g\cdot y)(a)-T(a)(x-a) 
=[g(x)-g(a)-Dg(a)(x-a)]y 
$$
for all $x\in E$, we conclude that 
$$
\lim_{x\to a}\frac{(g\cdot y)(x)-(g\cdot y)(a)-T(a)(x-a)}{\left\|x-a\right\|}=0. 
$$
Thus $g\cdot y$ is holomorphic at $a$ with $D(g\cdot y)(a)=yDg(a)$.
\end{proof}

We are now ready to establish the main result of this section.

\begin{proposition}\label{ideal summing} 
The space $(\mathcal{D}^{\H^\infty}_p(U,F),d^{\H^\infty}_p) $ is a regular Banach ideal of bounded holomorphic mappings. Furthermore, $\mathcal{D}^{\H^\infty}_1(U,F)=\H^\infty(U,F)$ with $d^{\H^\infty}_1(f)=\left\|f\right\|_\infty$ for all $f\in\mathcal{D}^{\mathcal{H}^\infty}_1(U,F)$.
\end{proposition}

\begin{proof}
We will only prove the case $1<p<\infty$. The cases $p=1$ and $p=\infty$ follow similarly.

(N1) We first show that $\mathcal{D}^{\H^\infty}_p(U,F)\subseteq\H^\infty(U,F)$ with $\left\|f\right\|_\infty\leq d^{\H^\infty}_p(f)$ for all $f\in\mathcal{D}^{\H^\infty}_p(U,F)$. Indeed, given $f\in\mathcal{D}^{\H^\infty}_p(U,F)$, we have 
$$
\left|\left\langle y^*,f(x)\right\rangle\right|\leq d^{\H^\infty}_p(f)\sup_{y^{**}\in B_{F^{**}}}\left|y^{**}(y^*)\right|=d^{\mathcal{H}^\infty}_p(f) 
$$
for all $x\in U$ and $y^*\in F^*$. By Hahn--Banach theorem, we obtain that $\left\|f(x)\right\|\leq d^{\H^\infty}_p(f)$ for all $x\in U$. Hence $f\in\H^\infty(U,F)$ with $\left\|f\right\|_\infty\leq d^{\mathcal{H}^\infty}_p(f)$.

Let $f_1,f_2\in\mathcal{D}^{\H^\infty}_p(U,F)$. Given $n\in\mathbb{N}$, $\lambda_1,\ldots,\lambda_n\in\mathbb{C}$, $x_1,\ldots,x_n\in U$ and $y^*_1,\ldots,y^*_n\in F^*$, we have 
\begin{align*}
\sum_{i=1}^n \left|\lambda_i\right|\left|\left\langle
y^*_i,f_1(x_i)\right\rangle\right| &\leq d^{\H^\infty}_p(f_1)\left(\sum_{i=1}^n\left|\lambda_i\right|^p\right)^{\frac{1}{p}}\sup_{y^{**}\in B_{F^{**}}}\left(\sum_{i=1}^n\left|y^{**}(y^*_i)\right|^{p^*}\right)^{\frac{1}{p^*}}, \\
\sum_{i=1}^n \left|\lambda_i\right|\left|\left\langle y^*_i,f_2(x_i)\right\rangle\right| &\leq d^{\H^\infty}_p(f_2)\left(\sum_{i=1}^n\left|\lambda_i\right|^p\right)^{\frac{1}{p}}\sup_{y^{**}\in B_{F^{**}}}\left(\sum_{i=1}^n\left|y^{**}(y^*_i)\right|^{p^*}\right)^{\frac{1}{p^*}}.
\end{align*}
Using 
the two inequalities above, we obtain 
\begin{align*}
\sum_{i=1}^n \left|\lambda_i\right|\left|\left\langle y^*_i,(f_1+f_2)(x_i)\right\rangle\right|&\leq\sum_{i=1}^n\left|\lambda_i\right|\left|\left\langle y^*_i,f_1(x_i)\right\rangle\right|+\sum_{i=1}^n\left|\lambda_i\right|\left|\left\langle y^*_i,f_2(x_i)\right\rangle\right| \\
&\leq\left(d^{\H^\infty}_p(f_1)+d^{\H^\infty}_p(f_2)\right)\left(\sum_{i=1}^n\left|\lambda_i\right|^p\right)^{\frac{1}{p}}\sup_{y^{**}\in B_{F^{**}}}\left(\sum_{i=1}^n\left|y^{**}(y^*_i)\right|^{p^*}\right)^{\frac{1}{p^*}}.
\end{align*}
This tells us that $f_1+f_2\in\mathcal{D}^{\H^\infty}_p(U,F)$ with $d^{\H^\infty}_p(f_1+f_2)\leq d^{\H^\infty}_p(f_1)+d^{\H^\infty}_p(f_2)$.

Let $\lambda\in\mathbb{C}$ and $f\in\mathcal{D}^{\H^\infty}_p(U,F)$. Given $n\in\mathbb{N}$, $\lambda_i\in\mathbb{C}$, $x_i\in U$ and $y^*_i\in F^*$ for $i=1,\ldots,n$, we have 
\begin{align*}
\sum_{i=1}^n\left|\lambda_i\right|\left|\left\langle y^*_i,(\lambda f)(x_i)\right\rangle\right| 
&=\left|\lambda\right|\sum_{i=1}^n\left|\lambda_i\right|\left|\left\langle y^*_i,f(x_i)\right\rangle\right| \\
&\leq\left|\lambda\right|d^{\H^\infty}_p(f)\left(\sum_{i=1}^n\left|\lambda_i\right|^p\right)^{\frac{1}{p}} \sup_{y^{**}\in B_{F^{**}}}\left(\sum_{i=1}^n\left|y^{**}(y^*_i)\right|^{p^*}\right)^{\frac{1}{p^*}},
\end{align*}
and thus $\lambda f\in\mathcal{D}^{\H^\infty}_p(U,F)$ with $d^{\H^\infty}_p(\lambda f)\leq|\lambda|d^{\H^\infty}_p(f)$.
This implies that $d^{\H^\infty}_p(\lambda
f)=0=\left|\lambda\right|d^{\H^\infty}_p(f)$ if $\lambda=0$. For $\lambda\neq 0$, we have $d^{\H^\infty}_p(f)=d^{\H^\infty}_p(\lambda^{-1}(\lambda f))\leq\left|\lambda\right|^{-1}d^{\mathcal{H}^\infty}_p(\lambda f)$, hence $\left|\lambda\right|d^{\H^\infty}_p(f)\leq d^{\H^\infty}_p(\lambda f)$, and so $d^{\mathcal{H}^\infty}_p(\lambda f)=\left|\lambda\right|d^{\H^\infty}_p(f)$.

Moreover, if $f\in\mathcal{D}^{\H^\infty}_p(U,F)$ and $d^{\mathcal{H}^\infty}_p(f)=0$, then $\left\|f\right\|_\infty=0$ by (N1), and so $f=0$.
Thus, $\left(\mathcal{D}^{\H^\infty}_p(U,F),d^{\H^\infty}_p\right)$ is a normed space.

To prove that $\left(\mathcal{D}^{\H^\infty}_p(U,F),d^{\H^\infty}_p\right)$ is complete, it suffices to prove that every absolutely convergent series is convergent. So let $(f_n)_{n\in\mathbb{N}}$ be a
sequence in $\mathcal{D}^{\H^\infty}_p(U,F)$ such that $\sum_{n\in\mathbb{N}}d^{\H^\infty}_p(f_n)$ is convergent. Since $\left\|f_n\right\|_\infty\leq d^{\H^\infty}_p(f_n)$ for all $n\in\mathbb{N}$ and $\left(\H^\infty(U,F),\left\|\cdot\right\|_\infty\right)$ is a Banach space, then the series $\sum_{n\in\mathbb{N}}f_n$ converges in $\left(\H^\infty(U,F),\left\|\cdot\right\|_\infty\right)$ to a function $f\in\H^\infty(U,F)$. Given $m\in\mathbb{N}$, $x_1,\ldots,x_m\in U$, $y^*_1,\ldots,y^*_m\in F^*$ and $\lambda_1,\ldots,\lambda_m\in\mathbb{C}$, we have 
\begin{align*}
\sum_{k=1}^m\left|\lambda_k\right|\left|\left\langle y^*_k,\sum_{i=1}^nf_i(x_k)\right\rangle\right| 
&\leq d^{\H^\infty}_p\left(\sum_{i=1}^nf_i\right)\left(\sum_{k=1}^m\left|\lambda_k\right|^p\right)^{\frac{1}{p}}\sup_{y^{**}\in B_{F^{**}}}\left(\sum_{k=1}^m\left|y^{**}(y^*_k)\right|^{p^*}\right)^{\frac{1}{p^*}} \\
&\leq\left(\sum_{i=1}^n d^{\H^\infty}_p(f_i)\right)\left(\sum_{k=1}^m\left|\lambda_k\right|^p\right)^{\frac{1}{p}}\sup_{y^{**}\in B_{F^{**}}}\left(\sum_{k=1}^m\left|y^{**}(y^*_k)\right|^{p^*}\right)^{\frac{1}{p^*}}
\end{align*}
for all $n\in\mathbb{N}$, and by taking limits with $n\to\infty$ yields 
$$
\sum_{k=1}^m\left|\lambda_k\right|\left|\left\langle y^*_k,f(x_k)\right\rangle\right|
\leq\left(\sum_{n=1}^\infty d^{\H^\infty}_p(f_n)\right) \left(\sum_{k=1}^m\left|\lambda_k\right|^p\right)^{\frac{1}{p}}\sup_{y^{**}\in B_{F^{**}}}\left(\sum_{k=1}^m\left|y^{**}(y^*_k)\right|^{p^*}\right)^{\frac{1}{p^*}}. 
$$
Hence $f\in\mathcal{D}^{\H^\infty}_p(U,F)$ with $\pi^{\H^\infty}_p(f)\leq\sum_{n=1}^\infty d^{\H^\infty}_p(f_n)$. Moreover, we have 
$$
d^{\H^\infty}_p\left(f-\sum_{i=1}^nf_i\right)=d^{\H^\infty}_p\left(\sum_{i=n+1}^\infty f_i\right)\leq\sum_{i=n+1}^\infty d^{\H^\infty}_p(f_i) 
$$
for all $n\in\mathbb{N}$, and thus $f$ is the $d^{\H^\infty}_p$-limit of the series $\sum_{n\in\mathbb{N}}f_n$.

(N2) Let $g\in\H^\infty(U)$ and $y\in F$. If $g=0$ or $y=0$, there is nothing to prove. Assume $g\neq 0$ and $y\neq 0$. By Lemma \ref{lem-3}, $g\cdot y\in\H^\infty(U,F)$. Given $n\in\mathbb{N}$, $x_1,\ldots,x_n\in U$, $y^*_1,\ldots,y^*_n\in F^*$ and $\lambda_1,\ldots,\lambda_n\in\mathbb{C}$, we have 
\begin{align*}
\sum_{i=1}^n \left|\lambda_i\right|\left|\left\langle y^*_i,(g\cdot y)(x_i)\right\rangle\right|
&=\left\|g\right\|_\infty\left\|y\right\|\sum_{i=1}^n \left|\lambda_i\right|\left|\left\langle y^*_i,\frac{g(x_i)y}{\left\|g\right\|_\infty\left\|y\right\|}\right\rangle\right| \\
&\leq\left\|g\right\|_\infty\left\|y\right\|\left(\sum_{i=1}^n\left|\lambda_i\right|^p\right)^{\frac{1}{p}}\left(\sum_{i=1}^n\left|\left\langle y^*_i,\frac{g(x_i)y}{\left\|g\right\|_\infty\left\|y\right\|}
\right\rangle\right|^{p^*}\right)^{\frac{1}{p^*}} \\
&=\left\|g\right\|_\infty\left\|y\right\|\left(\sum_{i=1}^n\left|\lambda_i\right|^p\right)^{\frac{1}{p}}\left(\sum_{i=1}^n\left|\left\langle \kappa_F\left(\frac{g(x_i)y}{\left\|g\right\|_\infty\left\|y\right\|}\right),y^*_i,\right\rangle\right|^{p^*}\right)^{\frac{1}{p^*}} \\
&\leq\left\|g\right\|_\infty\left\|y\right\|\left(\sum_{i=1}^n\left|\lambda_i\right|^p\right)^{\frac{1}{p}}\sup_{y^{**}\in B_{F^{**}}}\left(\sum_{i=1}^n\left|y^{**}(y^*_i)\right|^{p^*}\right)^{\frac{1}{p^*}}
\end{align*}
by applying the H\"older inequality, and therefore $g\cdot y\in\mathcal{D}^{\H^\infty}_p(U,F)$ with $d^{\H^\infty}_p(g\cdot y)\leq\left\|g\right\|_\infty\left\|y\right\|$.
Conversely, by applying what was proved in (N1), we have $\left\|g\right\|_\infty\left\|y\right\|=\left\|g\cdot y\right\|_\infty\leq d^{\H^\infty}_p(g\cdot y)$.

(N3) Let $H,G$ be complex Banach spaces, $V$ be an open subset of $H$, $h\in\H(V,U)$, $f\in\mathcal{D}_p^{\H^\infty}(U,F)$ and $S\in\L(F,G)$. We can suppose $S\neq 0$. Given $n\in\mathbb{N}$, $x_1,\ldots,x_n\in U$, $y^*_1,\ldots,y^*_n\in G^*$ and $\lambda_1,\ldots,\lambda_n\in\mathbb{C}$, we have 
\begin{align*}
\sum_{i=1}^n\left|\lambda_i\right|\left|\left\langle y^*_i,S(f(h(x_i)))\right\rangle\right|&=\sum_{i=1}^n\left|\lambda_i\right|\left|\left\langle y^*_i\circ S, f(h(x_i))\right\rangle\right|\\
&\leq d_p^{\H^\infty}(f)\left(\sum_{i=1}^n\left|\lambda_i\right|^p\right)^{\frac{1}{p}}\sup_{y^{**}\in B_{F^{**}}}\left(\sum_{i=1}^n\left|y^{**}(y^*_i\circ S)\right|^{p^*}\right)^{\frac{1}{p^*}}\\
&=\left\|S\right\|d_p^{\H^\infty}(f)\left(\sum_{i=1}^n\left|\lambda_i\right|^p\right)^{\frac{1}{p}}\sup_{y^{**}\in B_{F^{**}}}\left(\sum_{i=1}^n\left|\left(y^{**}\circ\frac{S^*}{\left\|S\right\|}\right)(y^*_i)\right|^{p^*}\right)^{\frac{1}{p^*}}\\
&\leq \left\|S\right\|d_p^{\H^\infty}(f)\left(\sum_{i=1}^n\left|\lambda_i\right|^p\right)^{\frac{1}{p}}\sup_{z^{**}\in B_{G^{**}}}\left(\sum_{i=1}^n\left|z^{**}(y^*_i)\right|^{p^*}\right)^{\frac{1}{p^*}}
\end{align*}
and therefore $S\circ f\circ h\in\mathcal{D}_p^{\H^\infty}(V,G)$ with $d_p^{\H^\infty}(S\circ f\circ h)\leq\left\|S\right\|d_p^{\H^\infty}(f)$.

We now prove that the ideal $\mathcal{D}_p^{\H^\infty}(U,F)$ is regular. Let $f\in\H^\infty(U,F)$ and assume that $\kappa_F\circ f\in\mathcal{D}^{\H^\infty}_p(U,F^{**})$. Given $n\in\mathbb{N}$, $x_1,\ldots,x_n\in U$, $y^*_1,\ldots,y^*_n\in F^*$ and $\lambda_1,\ldots,\lambda_n\in\mathbb{C}$, we have 
\begin{align*}
\sum_{i=1}^n\left|\lambda_i\right|\left|\left\langle y^*_i,f(x_i)\right\rangle\right|
&=\sum_{i=1}^n\left|\lambda_i\right|\left|\left\langle \kappa_F(f(x_i)),y^*_i\right\rangle\right| \\
&\leq d_p^{\H^\infty}(\kappa_F\circ f)\left(\sum_{i=1}^n\left|\lambda_i\right|^p\right)^{\frac{1}{p}}\sup_{y^{**}\in B_{F^{**}}}\left(\sum_{i=1}^n\left|y^{**}(y^*_i)\right|^{p^*}\right)^{\frac{1}{p^*}},
\end{align*}
and thus $f\in\mathcal{D}^{\H^\infty}_p(U,F)$ with $d^{\H^\infty}_p(f)\leq d^{\H^\infty}_p(\kappa_F\circ f)$. The converse inequality follows from (N3).

Finally, we have seen in (N1) that $\mathcal{D}^{\H^\infty}_1(U,F)\subseteq\H^\infty(U,F)$ with $\left\|f\right\|_\infty\leq d^{\H^\infty}_1(f)$ for all $f\in \mathcal{D}^{\H^\infty}_1(U,F)$. For the converse, let $f\in\H^\infty(U,F)$. If $f=0$, there is nothing to prove. Assume $f\neq 0$. Given $n\in\mathbb{N}$, $x_1,\ldots,x_n\in U$, $y^*_1,\ldots,y^*_n\in F^*$ and $\lambda_1,\ldots,\lambda_n\in\mathbb{C}$, we have 
\begin{align*}
\sum_{i=1}^n\left|\lambda_i\right|\left|\left\langle y^*_i,f(x_i)\right\rangle\right|
&=\left\|f\right\|_\infty\sum_{i=1}^n\left|\lambda_i\right|\left|\left\langle \kappa_F\left(\frac{f(x_i)}{\left\|f\right\|_\infty}\right),y^*_i\right\rangle\right| \\
&\leq\left\|f\right\|_\infty\left(\sum_{i=1}^n\left|\lambda_i\right|\right)\max_{1\leq i\leq n}\left( \sup_{y^{**}\in B_{F^{**}}}\left|y^{**}(y^*_i)\right|\right) \\
&=\left\|f\right\|_\infty\left(\sum_{i=1}^n\left|\lambda_i\right|\right)\sup_{y^{**}\in B_{F^{**}}}\left(\max_{1\leq i\leq n}\left|y^{**}(y^*_i)\right|\right) ,
\end{align*}
and therefore $f\in\mathcal{D}^{\H^\infty}_1(U,F)$ with $d^{\H^\infty}_1(f)\leq\left\|f\right\|_\infty$.
\end{proof}


\section{The tensor product $\Delta(U)\otimes F$}\label{3}

The elements of the tensor product of two linear spaces can be viewed as linear mappings or bilinear forms (see \cite[Chapter 1]{Rya-02}). Following this idea, 
we introduce the tensor product $\Delta(U)\otimes F$ as a space of linear functionals on $\H^\infty(U,F^*)$.

\begin{definition}\label{def-Hol-tensor-product} 
Let $E$ and $F$ be complex Banach spaces and let $U$ be an open subset of $E$. For each $x\in U$, let $\delta(x)\colon\H^\infty(U)\to\mathbb{C}$ be the linear functional defined by 
$$
\left\langle \delta(x),f\right\rangle=f(x) \qquad \left(f\in\H^\infty(U)\right). 
$$
Let $\Delta(U)$ be the linear subspace of $\H^\infty(U)^{\prime }$ spanned by the set 
$$
\left\{\delta(x)\colon x\in U\right\}. 
$$
For any $x\in U$ and $y\in F$, let $\delta(x)\otimes y\colon\H^\infty(U,F^*)\to\mathbb{C}$ be the linear functional given by 
$$
\left(\delta(x)\otimes y\right)(f)=\left\langle f(x),y\right\rangle\qquad \left(f\in\H^\infty(U,F^*)\right) . 
$$
We define the tensor product $\Delta(U)\otimes F$ as the linear subspace of $\H^\infty(U,F^*)^{\prime }$ spanned by the set 
$$
\left\{\delta(x)\otimes y\colon x\in U,\, y\in F\right\}. 
$$
\end{definition}



We say that $\delta(x)\otimes y$ is an \textit{elementary tensor} of $\Delta(U)\otimes F$. Note that each element $u$ in $\Delta(U)\otimes F$ is of the form $u=\sum_{i=1}^n\lambda_i(\delta(x_i)\otimes y_i)$, where $n\in\mathbb{N}$, $\lambda_i\in\mathbb{C}$, $x_i\in U$ and $y_i\in F$ for $i=1,\ldots,n$. This representation of $u$ is not unique.

It is worth noting that each element $u$ of $\Delta(U)\otimes F$ can be represented as $u=\sum_{i=1}^n \delta(x_i)\otimes y_i$ since $\lambda(\delta(x)\otimes y)=\delta(x)\otimes(\lambda y)$. This
representation of $u$ admits the following refinement (see \cite[p. 3]{Rya-02}).

\begin{lemma}
\label{lem-0} Every nonzero tensor $u\in\Delta(U)\otimes F$ has a representation in the form 
$$
u=\sum_{i=1}^m\delta(z_i)\otimes d_i, 
$$
where 
$$
m=\min\left\{k\in\mathbb{N}\colon \exists z_1,\ldots,z_k\in U,\ d_1,\ldots,d_k\in F \; | \; u=\sum_{i=1}^k\delta(z_i)\otimes d_i\right\}, 
$$
and the sets $\left\{\delta(z_1),\ldots,\delta(z_m)\right\}\subseteq\Delta(U)$ and $\left\{d_1,\ldots,d_m\right\}\subseteq F$ are both linearly independent. $\hfill\Box$
\end{lemma}

As a straightforward consequence from Definition \ref{def-Hol-tensor-product}, we describe the action of a tensor $u$ in $\Delta(U)\otimes F$ on a function $f$ in $\H^\infty(U,F^*)$:

\begin{lemma}\label{lem-2}
Let $u=\sum_{i=1}^n \lambda_i\delta(x_i)\otimes y_i\in\Delta(U)\otimes F$ and $f\in\H^\infty(U,F^*)$. Then 
$$
u(f)=\sum_{i=1}^n\lambda_i\left\langle f(x_i),y_i\right\rangle . 
$$
$\hfill\Box$
\end{lemma}

Our next aim is to characterize the zero tensor of $\Delta(U)\otimes F$. Compare to \cite[Proposition 1.2]{Rya-02}.

\begin{proposition}
\label{pro-0} If $u=\sum_{i=1}^n\delta(x_i)\otimes y_i\in\Delta(U)\otimes F$, the following are equivalent:
\begin{enumerate}
\item $u=0$. 
\item $\sum_{i=1}^ng(x_i)\phi(y_i)=0$ for every $g\in B_{\H^\infty(U)}$ and $\phi\in B_{F^*}$.
\end{enumerate}
\end{proposition}

\begin{proof}
$(i) \Rightarrow (ii)$: If $u=0$, then $u(f)=0$ for all $f\in\H^\infty(U,F^*)$. Since $u=\sum_{i=1}^n \delta(x_i)\otimes y_i$, it follows that $\sum_{i=1}^n\langle f(x_i),y_i\rangle=0$ for all $f\in\H^\infty(U,F^*)$ by Lemma \ref{lem-2}. For any $g\in B_{\H^\infty(U)} $ and $\phi\in B_{F^*}$, the function $g\cdot\phi$ is in $\H^\infty(U,F^*)$ by Lemma \ref{lem-3}, and therefore we have 
$$
\sum_{i=1}^n g(x_i)\phi(y_i) =\sum_{i=1}^n\left\langle g(x_i)\phi,y_i\right\rangle=\sum_{i=1}^n\left\langle(g\cdot\phi)(x_i),y_i\right\rangle =0. 
$$

$(ii) \Rightarrow (i)$: By Lemma \ref{lem-0}, we can write $u=\sum_{i=1}^m\delta(z_i)\otimes d_i$, where the vectors $d_i$ in $F$ are linearly independent. It follows that 
$$
\sum_{i=1}^n \delta(x_i)\otimes y_i+\sum_{i=1}^m\delta(z_i)\otimes (-d_i)=u-u=0, 
$$
and, by using the fact proved that (i) implies (ii), we have 
$$
\sum_{i=1}^n g(x_i)\phi(y_i)+\sum_{i=1}^m g(z_i)\phi(-d_i)=0 
$$
for all $g\in B_{\H^\infty(U)}$ and $\phi\in B_{F^*}$. If (ii) holds, we get that 
$$
\sum_{i=1}^mg(z_i)\phi(d_i)=\sum_{i=1}^n g(x_i)\phi(y_i)=0 
$$
for all $g\in B_{\H^\infty(U)}$ and $\phi\in B_{F^*}$. Let $J_U(g)$ be the functional in $\G^\infty(U)^*$ considered in Theorem \ref{teo1}. We have 
$$
\sum_{i=1}^m \left\langle J_U(g),\delta(z_i)\right\rangle\phi(d_i)=\sum_{i=1}^mg(z_i)\phi(d_i)=0 
$$
for all $g\in B_{\H^\infty(U)}$ and $\phi\in B_{F^*}$. Since the mapping $J_U$ is an isometric isomorphism from $\H^\infty(U)$ onto $\mathcal{G}^\infty(U)^*$ by Theorem \ref{teo1}, it follows
that 
$$
\gamma\left(\sum_{i=1}^m\delta(z_i)\phi(d_i)\right)=\sum_{i=1}^m\gamma(\delta(z_i))\phi(d_i)=0 
$$
for all $\gamma\in B_{\mathcal{G}^\infty(U)^*}$ and $\phi\in B_{F^*}$. As $B_{\mathcal{G}^\infty(U)^*}$ separates points of $\mathcal{G}^\infty(U)$, this implies that 
$$
\sum_{i=1}^m\delta(z_i)\phi(d_i)=0 
$$
for all $\phi\in B_{F^*}$. Moreover, $\left\{d_1,\ldots,d_m\right\}$ are linearly independent in $F$, the Hahn--Banach theorem provides, for each $j\in\{1,\ldots,m\}$, a functional $\phi_j\in B_{F^*}$ such that $\phi_j(d_j)=1$ and $\phi_j(d_i)=0$ for all $i\in\{1,\ldots,m\}\setminus\{j\}$. Hence $0=\sum_{i=1}^m\delta(z_i)\phi_j(d_i)=\delta(z_j)$ for each $j\in\{1,\ldots,m\}$ and thus $u=0$.
\end{proof}

By Definition \ref{def-Hol-tensor-product}, $\Delta(U)\otimes F$ is a linear subspace of $\H^\infty(U,F^*)^{\prime }$. In fact, we have:

\begin{proposition}\label{theo-dual-pair} 
$\left\langle \Delta(U)\otimes F,\H^\infty(U,F^*)\right\rangle$ forms a dual pair, where the bilinear form $\left\langle \cdot,\cdot\right\rangle$ associated to the dual pair is given
by 
$$
\left\langle u,f\right\rangle=\sum_{i=1}^n \lambda_i\left\langle f(x_i),y_i\right\rangle 
$$
for $u=\sum_{i=1}^n\lambda_i \delta(x_i)\otimes y_i\in\Delta(U)\otimes F$ and $f\in\H^\infty(U,F^*)$.
\end{proposition}

\begin{proof}
Since $\langle u,f\rangle=u(f)$ by Lemma \ref{lem-2}, it is immediate that $\langle\cdot,\cdot\rangle$ is a well-defined bilinear map from $(\Delta(U)\otimes F)\times\H^\infty(U,F^*)$ to $\mathbb{C}$. On the one hand, if $u\in\Delta(U)\otimes F$ and $\langle u,f\rangle=0$ for all $f\in\H^\infty(U,F^*)$, then $u=0$ follows easily from Proposition \ref{pro-0}, and thus $\H^\infty(U,F^*)$ separates points of $\Delta(U)\otimes F$. On the other hand, if $f\in\H^\infty(U,F^*)$ and $\langle u,f\rangle=0$ for all $u\in \Delta(U)\otimes F$, then $\left\langle f(x),y\right\rangle=\left\langle\delta(x)\otimes y,f\right\rangle=0$ for all $x\in U$ and $y\in F$, this means that $f=0$ and thus $\Delta(U)\otimes F$ separates points of $\H^\infty(U,F^*)$. 
\end{proof}

Since $\left\langle\Delta(U)\otimes F,\H^\infty(U,F^*)\right\rangle$ is a dual pair, we can identify $\H^\infty(U,F^*)$ with a linear subspace of $(\Delta(U)\otimes F)^{\prime }$ as follows.

\begin{corollary}
\label{linearization} For each $f\in\H^\infty(U,F^*)$, the functional $\Lambda_0(f)\colon \Delta(U)\otimes F\to\mathbb{C}$, given by 
$$
\Lambda_0(f)(u)=\sum_{i=1}^n\lambda_i\left\langle f(x_i),y_i\right\rangle 
$$
for $u=\sum_{i=1}^n \lambda_i\delta(x_i)\otimes y_i\in \Delta(U)\otimes F$, is linear. We will say that $\Lambda_0(f)$ is the linear functional on $\Delta(U)\otimes F$ associated to $f$. Furthermore, the map $f\mapsto \Lambda_0(f)$ is a linear monomorphism from $\H^\infty(U,F^*)$ into $(\Delta(U)\otimes F)^{\prime }$.
\end{corollary}

\begin{proof}
Let $f\in\H^\infty(U,F^*)$. Note that $\Lambda_0(f)(u)=\left\langle u,f\right\rangle$ for all $u\in\Delta(U)\otimes F$. It is immediate that $\Lambda_0(f)$ is a well-defined linear functional on $\Delta(U)\otimes F$ and that $f\mapsto \Lambda_0(f)$ from $\H^\infty(U,F^*)$ into $(\Delta(U)\otimes F)^{\prime }$ is a well-defined linear map. Finally, let $f\in\H^\infty(U,F^*)$ and assume that $\Lambda_0(f)=0$. Then $\left\langle u,f\right\rangle=0$ for all $u\in\Delta(U)\otimes F$. Since $\Delta(U)\otimes F$ separates points of $\H^\infty(U,F^*)$ by Proposition \ref{theo-dual-pair}, it follows that $f=0$ and this proves that $\Lambda_0$ is one-to-one.
\end{proof}

Next, we will introduce several useful norms on the tensor product $\Delta(U)\otimes F$. We begin with the dual norm induced by the supremum norm of $\H^\infty(U,F^*)$ that, as we will see, coincides with the projective norm. 

Given two linear spaces $E$ and $F$, the tensor product space $E\otimes F$ equipped with a norm $\alpha$ will be denoted by $E\otimes_\alpha F$, and the completion of $E\otimes_\alpha F$ by $E\widehat{\otimes}_\alpha F$. A \textit{cross-norm} on $E\otimes F$ is a norm $\alpha$ such that $\alpha(x\otimes y)=\left\|x\right\|\left\|y\right\|$ for all $x\in E$ and $y\in F$. 

\begin{theorem}
\label{teo-L} The linear space $\Delta(U)\otimes F$ is contained in $\H^\infty(U,F^*)^*$ and the norm $H$ on $\Delta(U)\otimes F$ induced by the dual norm of the norm $\left\|\cdot\right\|_\infty$ of $\H^\infty(U,F^*)$, given by 
$$
H(u)=\sup\left\{\left|u(f)\right|\colon f\in\H^\infty(U,F^*),\ \left\|f\right\|_\infty\leq 1 \right\}\qquad \left(u\in \Delta(U)\otimes F\right), 
$$
is a cross-norm on $\Delta(U)\otimes F$.
\end{theorem}

\begin{proof}
Let $\lambda\in\mathbb{C}$, $x\in U$ and $y\in F$. Since $\lambda\delta(x)\otimes y$ is a linear map on $\H^\infty(U,F^*)$ and 
$$
\left|(\lambda\delta(x)\otimes y)(f)\right| =\left|\lambda\left\langle f(x),y\right\rangle\right|
\leq\left|\lambda\right|\left\|f(x)\right\|\left\|y\right\|
\leq\left|\lambda\right|\left\|f\right\|_\infty\left\|y\right\| 
$$
for all $f\in\H^\infty(U,F^*)$, then $\lambda\delta(x)\otimes y\in\H^\infty(U,F^*)^*$ with $\left\|\lambda\delta(x)\otimes y\right\|\leq \left|\lambda\right|\left\|y\right\|$, and thus $\Delta(U)\otimes F\subseteq\H^\infty(U,F^*)^*$.

We now prove that $H$ is a cross-norm on $\Delta(U)\otimes F$. By the above proof, we have $\left|(\lambda\delta(x)\otimes y)(f)\right|\leq\left|\lambda\right|\left\|y\right\|$ for all $f\in\H^\infty(U,F^*)$ with $\left\|f\right\|_\infty\leq 1$, and hence $H(\lambda\delta(x)\otimes y)\leq\left|\lambda\right|\left\|y\right\|=\left\|\lambda\delta(x)\right\|\left\|y\right\|$. For the reverse estimate, take $\phi\in F^*$ with $\left\|\phi\right\|=1$ such that $\left|\left\langle\phi,y\right\rangle\right|=\left\|y\right\|$, and define the constant mapping $f\colon U\to F^*$ by $f(z)=\phi$ for all $z\in U$. Clearly, $f\in\H^\infty(U,F^*)$ with $\left\|f\right\|_\infty=1$ and 
$$
\left|(\lambda\delta(x)\otimes y)(f)\right|
=\left|\lambda\right|\left|\left\langle f(x),y\right\rangle\right|
=\left|\lambda\right|\left|\left\langle \phi,y\right\rangle\right|
=\left|\lambda\right|\left\|y\right\|, 
$$
and therefore $\left\|\lambda\delta(x)\right\|\left\|y\right\|=\left|\lambda\right|\left\|y\right\| 
\leq H(\lambda\delta(x)\otimes y)$.
\end{proof}

Given two Banach spaces $E$ and $F$, the projective norm $\pi$ on $E\otimes F$ (see \cite[Chapter 2]{Rya-02}) takes the following form on $\Delta(U)\otimes F$:
$$
\pi(u)=\inf\left\{\sum_{i=1}^n\left|\lambda_i\right|\left\|y_i\right\|\colon u=\sum_{i=1}^n\lambda_i\delta(x_i)\otimes y_i\right\}\qquad (u\in \Delta(U)\otimes F).
$$

We next see that, on the space $\Delta(U)\otimes F$, the projective norm and the norm induced by the dual norm of the supremun norm of $\H^\infty(U,F^*)$ coincide.

\begin{proposition}\label{pies L} 
$\pi(u)=H(u)$ for all $u\in \Delta(U)\otimes F$.
\end{proposition}

\begin{proof}
Since $\pi$ is the greatest cross-norm on $\Delta(U)\otimes F$ (see \cite[pp. 15-16]{Rya-02}), we have $H\leq\pi$ by Theorem \ref{teo-L}. To prove that $H\geq\pi$, suppose by contradiction that $H(u_0)<1<\pi(u_0)$ for some $u_0\in \Delta(U)\otimes F$. Denote $B=\{u\in \Delta(U)\otimes F\colon \pi(u)\leq 1\}$. Clearly, $B$ is a closed and convex set in $\Delta(U)\otimes_\pi F$. Applying the Hahn--Banach separation theorem to $B$ and $\{u_0\}$, we obtain a functional $\eta\in(\Delta(U)\otimes_\pi F)^*$ such that 
$$
1=\|\eta\|=\sup\{\mathrm{Re}\,\eta(u)\colon u\in B \} < \mathrm{Re}\,\eta(u_0). 
$$
Define $f\colon U\to F^*$ by $\langle f(x),y\rangle=\eta\left(\delta(x)\otimes y\right)$ for all $y\in F$ and $x\in U$. It is easy to prove that $f$ is well defined and $f\in\H^\infty(U,F^*)$ with $\left\|f\right\|_\infty\leq 1$. Moreover, $u(f)=\eta(u)$ for all $u\in\Delta(U)\otimes F$. Therefore $H(u_0)\geq |u_0(f)|\geq \mathrm{Re}\,u_0(f)=\mathrm{Re}\,\eta(u_0)$, so $H(u_0)>1$ and this is a
contradiction.
\end{proof}

We now will define the Chevet--Saphar norms on the tensor product $E\otimes F$. Let $E$ and $F$ be normed spaces and let $1\leq p\leq\infty$. Given $u=\sum_{i=1}^n x_i\otimes y_i\in E\otimes F$, denote 
$$
\left\|(x_1,\ldots,x_n)\right\|_{\ell^n_p(E)}
=\left\{\begin{array}{lll}
\displaystyle{\left(\sum_{i=1}^n\left\|x_i\right\|^p\right)^{\frac{1}{p}}} & \text{ if } & 1\leq p<\infty , \\ 
&  &  \\ 
\displaystyle\max_{1\leq i\leq n}\left\|x_i\right\| & \text{ if } & p=\infty,
\end{array}\right. 
$$
and 
$$\left\|(y_1,\ldots, y_n)\right\|_{\ell^{n,w}_{p}(F)} 
=\left\{\begin{array}{lll}
\displaystyle\sup_{y^*\in B_{F^*}}\displaystyle\left(\sum_{i=1}^n\left|y^*(y_i)\right|^p\right)^{\frac{1}{p}} & \text{if} & 1\leq p<\infty ,\\ 
&  &  \\ 
\displaystyle\sup_{y^*\in B_{F^*}}\displaystyle\left(\max_{1\leq i\leq n}\left|y^*(y_i)\right|\right) & \text{ if } & p=\infty .
\end{array}\right. 
$$
If $E=F=\mathbb{C}$, we write $\ell^n_p(E)=\ell^n_p$ and $\ell^{n,w}_{p^*}(F)=\ell^{n,w}_{p^*}$. According to \cite[Section 6.2]{Rya-02}, the Chevet--Saphar norms are defined on $E\otimes F$ by 
\begin{align*}
d_p(u)&=\inf\left\{\left\|(x_1,\ldots,x_n)\right\|_{\ell^{n,w}_{p^*}(E)}\left\|(y_1,\ldots, y_n)\right\|_{\ell^n_p(F)}\right\}, \\
g_p(u)&=\inf\left\{\left\|(x_1,\ldots,x_n)\right\|_{\ell^n_p(E)}\left\|(y_1,\ldots, y_n)\right\|_{\ell^{n,w}_{p^*}(F)}\right\},
\end{align*}
the infimum being extended over all representations of $u$ of the form $u=\sum_{i=1}^nx_i\otimes y_i\in E\otimes F$. 

Since $\left\|\delta(x)\right\|=1$ for all $x\in U$, the norm $g_p$ on $\Delta(U)\otimes F$ takes the form:
$$
g_p(u)=\inf\left\{\left\|(\lambda_1,\ldots,\lambda_n)\right\|_{\ell^n_p}\left\|(y_1,\ldots, y_n)\right\|_{\ell^{n,w}_{p^*}(F)}\colon u=\sum_{i=1}^n\lambda_i\delta(x_i)\otimes y_i\right\}. 
$$
Notice that $g_p$ is a cross-norm on $\Delta(U)\otimes F$.

We next show that the norm $g_1$ on $\Delta(U)\otimes F$ is justly the projective tensor norm $\pi$.

\begin{proposition}\label{new-now} 
$g_1(u)=\pi(u)$ for all $u\in\Delta(U)\otimes F$.
\end{proposition}

\begin{proof}
Let $u\in\Delta(U)\otimes F$ and let $\sum_{i=1}^n\lambda_i\delta(x_i)\otimes y_i$ be a representation of $u$. We have 
\begin{align*}
\pi(u) &\leq\sum_{i=1}^n |\lambda_i|\left\|y_i\right\|=\sum_{i=1}^n|\lambda_i|\left(\sup_{y^*\in B_{F^*}}\left|y^*(y_i)\right|\right) \\
&\leq\sum_{i=1}^n|\lambda_i|\max_{1\leq i\leq n}\left(\sup_{y^*\in B_{F^*}}\left|y^*(y_i)\right|\right)=\left\|(\lambda_1,\ldots,\lambda_n)\right\|_{\ell^n_1}\left\|(y_1,\ldots,y_n)\right\|_{\ell^{n,w}_\infty(F)},
\end{align*}
and taking the infimum over all representations of $u$ gives $\pi(u)\leq g_1(u)$. For the converse inequality, notice that $g_1(\lambda\delta(x)\otimes y)\leq|\lambda|\left\|y\right\|$ for all $\lambda\in\mathbb{C}$, $x\in U$ and $y\in F$. Since $g_1$ is a norm on $\Delta(U)\otimes F$, it follows that 
$$
g_1(u) =g_1\left(\sum_{i=1}^n\lambda_i\delta(x_i)\otimes y_i\right)
\leq\sum_{i=1}^ng_1\left(\lambda_i\delta(x_i)\otimes y_i\right)
\leq\sum_{i=1}^n\left|\lambda_i\right|\left\|y_i\right\| 
$$
and taking the infimum over all representations of $u$ yields $g_1(u)\leq\pi(u)$. 
\end{proof}


\section{Duality for Cohen strongly $p$-summing holomorphic mappings}\label{4}

We show now that the duals of the tensor product $\G^\infty(U)\widehat{\otimes}_{g_p} F$ can be canonically identified as spaces of Cohen strongly $p$-summing holomorphic mappings.

\begin{theorem}\label{messi-10}
\label{theo-dual-classic-00} Let $1\leq p\leq\infty$. Then $\mathcal{D}^{\H_\infty}_p(U,F^*)$ is isometrically isomorphic to $(\G^\infty(U)\widehat{\otimes}_{g_p} F)^*$, via the mapping $\Lambda\colon \mathcal{D}^{\H_\infty}_p(U,F^*)\to(\G^\infty(U)\widehat{\otimes}_{g_p} F)^*$ defined by 
$$
\Lambda(f)(u)=\sum_{i=1}^n\lambda_i\left\langle f(x_i),y_i\right\rangle 
$$
for $f\in\mathcal{D}^{\H_\infty}_p(U,F^*)$ and $u=\sum_{i=1}^n\lambda_i\delta(x_i)\otimes y_i\in\Delta(U)\otimes F$. Furthermore, its inverse 
is given by 
$$
\left\langle \Lambda^{-1}(\varphi)(x),y\right\rangle=\left\langle\varphi,\delta(x)\otimes y\right\rangle 
$$
for $\varphi\in(\G^\infty(U)\widehat{\otimes}_{g_p} F)^*$, $x\in U$ and $y\in F$.
\end{theorem}

\begin{proof}
We prove it for $1<p\leq\infty$. The case $p=1$ is similarly proved.

Let $f\in\mathcal{D}^{\H_\infty}_p(U,F^*)$ and let $\Lambda_0(f)\colon \Delta(U)\otimes F\to\mathbb{C}$ be its associate linear functional. We claim that $\Lambda_0(f)\in(\Delta(U)\otimes_{g_p} F)^*$ with 
$\left\|\Lambda_0(f)\right\|\leq d^{\H^\infty}_p(f)$. Indeed, given $u=\sum_{i=1}^n\lambda_i\delta(x_i)\otimes y_i\in \Delta(U)\otimes F$, we have 
\begin{align*}
\left|\Lambda_0(f)(u)\right|&=\left|\sum_{i=1}^n\lambda_i\left\langle f(x_i), y_i\right\rangle\right| \\
&\leq\sum_{i=1}^n\left|\lambda_i\right|\left|\left\langle\kappa_F(y_i),f(x_i)\right\rangle\right| \\
&\leq d^{\H^\infty}_p(f)\left\|(\lambda_1,\ldots,\lambda_n)\right\|_{\ell^n_p}\displaystyle\sup_{y^{***}\in B_{F^{***}}}\displaystyle\left(\sum_{i=1}^n\left|y^{***}(\kappa_F(y_i))\right|^{p^*}\right)^{\frac{1}{p^*}} \\
&\leq d^{\H^\infty}_p(f)\left\|(\lambda_1,\ldots,\lambda_n)\right\|_{\ell^n_p}\displaystyle\sup_{y^{*}\in B_{F^{*}}}\displaystyle\left(\sum_{i=1}^n\left|y^{*}(y_i)\right|^{p^*}\right)^{\frac{1}{p^*}} \\
&=d^{\H^\infty}_p(f)\left\|(\lambda_1,\ldots,\lambda_n)\right\|_{\ell^n_p}\left\|(y_1,\ldots,y_n)\right\|_{\ell^{n,w}_{p^*}(F)},
\end{align*}
and taking infimum over all the representations of $u$, we deduce that $\left|\Lambda_0(f)(u)\right|\leq d^{\H^\infty}_p(f)g_p(u)$. Since $u$ was arbitrary, then $\Lambda_0(f)$ is continuous on $\Delta(U)
\otimes_{g_p} F$ with $\left\|\Lambda_0(f)\right\|\leq d^{\H^\infty}_p(f)$, as claimed.

Since $\Delta(U)$ is a norm-dense linear subspace of $\G^\infty(U)$ and $g_p$ is a cross-norm on $\G^\infty(U)\otimes F$, then $\Delta(U)\otimes F$ is a dense linear subspace of $\G^\infty(U)\otimes_{g_p} F$ and therefore also of its completion $\G^\infty(U)\widehat{\otimes}_{g_p} F$. Hence there is a unique continuous mapping $\Lambda(f)$ from $\G^\infty(U)\widehat{\otimes}_{g_p} F$ to $%
\mathbb{C}$ that extends $\Lambda_0(f)$. Further, $\Lambda(f)$ is linear and $\left\|\Lambda(f)\right\|=\left\|\Lambda_0(f)\right\|$.

Let $\Lambda\colon\mathcal{D}^{\H_\infty}_p(U,F^*)\to(\G^\infty(U)\widehat{\otimes}_{g_p} F)^*$ be the mapping so defined. Since $\Lambda_0\colon\mathcal{D}^{\H_\infty}_p(U,F^*)\to(\Delta(U)\otimes F)'$ is a linear monomorphism by Corollary \ref{linearization}, it follows easily that $\Lambda$ is so. To prove that $\Lambda$ is a surjective isometry, let $\varphi\in(\G^\infty(U)\widehat{\otimes}_{g_p} F)^*$ and define $f_\varphi\colon U\to F^*$ by 
$$
\left\langle f_\varphi(x),y\right\rangle=\varphi(\delta(x)\otimes y)\qquad\left(x\in U,\; y\in F\right). 
$$
Given $x\in U$, the linearity of both $\varphi$ and the product tensor in the second variable yields that the functional $f_\varphi(x)\colon F\to\mathbb{C}$ is linear, and since 
$$
\left|\left\langle f_\varphi(x),y\right\rangle\right|
=\left|\varphi(\delta(x)\otimes y)\right|
\leq\left\|\varphi\right\|g_p(\delta(x)\otimes y)
\leq\left\|\varphi\right\|\left\|y\right\| 
$$
for all $y\in F$, we deduce that $f_\varphi(x)\in F^*$ with $||f_\varphi(x)||\leq \left\|\varphi\right\|$. Since $x$ was arbitrary, we have that $f_\varphi$ is bounded with $\left\|f_\varphi\right\|_\infty\leq\left\|\varphi\right\|$.

We now prove that $f_\varphi\colon U\to F^*$ is holomorphic. To this end, we first claim that, for every $y\in F$, the function $f_y\colon U\to\mathbb{C}$ defined by 
$$
f_y(x)=\varphi(\delta(x)\otimes y)\qquad (x\in U) 
$$
is holomorphic. Let $a\in U$. Since $g_U\colon U\to\G^\infty(U)$ is holomorphic by Theorem \ref{teo1}, there exists $Dg_U(a)\in\mathcal{L}(E,\G^\infty(U))$ such that 
$$
\lim_{x\to a}\frac{\delta(x)-\delta(a)-Dg_U(a)(x-a)}{\left\|x-a\right\|}=0. 
$$
Consider the function $T(a)\colon E\to\mathbb{C}$ given by 
$$
T(a)(x)=\varphi(Dg_U(a)(x)\otimes y)\qquad\left(x\in E\right). 
$$
Clearly, $T(a)\in E^*$ and since 
\begin{align*}
f_y(x)-f_y(a)-T(a)(x-a)
&=\varphi(\delta(x)\otimes y)-\varphi(\delta(a)\otimes y)-\varphi(Dg_U(a)(x-a)\otimes y) \\
&=\varphi\left((\delta(x)-\delta(a)-Dg_U(a)(x-a))\otimes y\right),
\end{align*}
it follows that 
\begin{align*}
\lim_{x\to a}\frac{f_y(x)-f_y(a)-T(a)(x-a)}{\left\|x-a\right\|}
&=\lim_{x\to a}\frac{\varphi\left((\delta(x)-\delta(a)-Dg_U(a)(x-a))\otimes y\right)}{\left\|x-a\right\|} \\
&=\lim_{x\to a}\varphi\left(\frac{\delta(x)-\delta(a)-Dg_U(a)(x-a)}{\left\|x-a\right\|}\otimes y\right) \\
&=\varphi(0\otimes y)=\varphi(0)=0.
\end{align*}
Hence $f_y$ is holomorphic at $a$ with $Df_y(a)=T(a)$, and this proves our claim. Now, notice that the set $\left\{\kappa_F(y)\colon y\in B_F\right\}\subseteq B_{F^{**}}$ is norming for $F^*$ since 
$$
\left\|y^*\right\|=\sup\left\{\left|y^*(y)\right|\colon y\in B_F\right\}=\sup\left\{\left|\kappa_F(y)(y^*)\right|\colon y\in B_F\right\} 
$$
for every $y^*\in F^*$, and that $\kappa_F(y)\circ f_\varphi=f_y$ for every $y\in F$ since 
$$
(\kappa_F(y)\circ f_\varphi)(x)=\kappa_F(y)(f_\varphi(x))=\left\langle f_\varphi(x),y\right\rangle=\varphi(\delta(x)\otimes y)=f_y(x) 
$$
for all $x\in U$.

We are now ready to show that $f_\varphi\colon U\to F^*$ is holomorphic.
Indeed, let $a\in U$ and $b\in E$. Denote $V=\left\{\lambda\in\mathbb{C}%
\colon a+\lambda b\in U\right\}$. Clearly, the mapping $h\colon V\to U$
given by $h(\lambda)=a+\lambda b$ is holomorphic. Since $f_\varphi\circ h$
is locally bounded and $\kappa_F(y)\circ(f_\varphi\circ h)=f_y\circ h$ is
holomorphic on the open set $V\subseteq\mathbb{C}$ for all $y\in F$,
Proposition A.3 in \cite{AreBatHieNeu-01} assures that $f_\varphi\circ h$ is
holomorphic. This means that $f_\varphi$ is G-holomorphic but since it is
also locally bounded, we deduce that $f_\varphi$ is continuous by \cite[%
Proposition 8.6]{Muj-86}. Now, we conclude that $f_\varphi$ is holomorphic
by Theorem \ref{teo-0}. 

We now prove that $f_\varphi\in\mathcal{D}^{H^\infty}_p(U,F^*)$. For it,
take $n\in\mathbb{N}$, $\lambda_i\in\mathbb{C}$, $x_i\in U$ and $y^{**}_i\in
F^{**}$ for $i=1,\ldots,n$. Let $\varepsilon>0$ and consider the
finite-dimensional subspaces $V=\mathrm{lin}\{y^{**}_1,\ldots,y^{**}_n\}%
\subseteq F^{**}$ and $W=\mathrm{lin}\{f_\varphi(x_1),\ldots,f_\varphi(x_n)%
\}\subseteq F^{*}$. The principle of local reflexivity \cite[Theorem 8.16]{DisJarTon-95} gives us a bounded linear operator $T_{(\varepsilon, V,W)}\colon V\to F$ such that
\begin{enumerate}
\item[i)] $T_{(\varepsilon, V,W)}(y^{**})=y^{**}$ for every $y^{**}\in V\cap\kappa_F(F)$,
\item[ii)] $(1-\varepsilon)\left\|y^{**}\right\|\leq\left\|T_{(\varepsilon,V,W)}(y^{**})\right\|\leq(1+\varepsilon)\left\|y^{**}\right\|$ for every $y^{**}\in V$, 
\item[iii)] $\left\langle y^*,T_{(\varepsilon, V,W)}(y^{**})\right\rangle=\left\langle y^{**},y^* \right\rangle$ for every $y^{**}\in V$ and $y^*\in W$.
\end{enumerate}
Using iii) and taking $y_i=T_{(\varepsilon, V,W)}(y_i^{**})$, we first have 
\begin{align*}
\left|\sum_{i=1}^n\lambda_i\left\langle y_i^{**},f_\varphi(x_i)\right\rangle\right|
&=\left|\sum_{i=1}^n\lambda_i\left\langle f_\varphi(x_i),T_{(\varepsilon,V,W)}(y_i^{**})\right\rangle\right| \\
&=\left|\sum_{i=1}^n\lambda_i\left\langle f_\varphi(x_i),y_i\right\rangle\right| \\
&=\left|\varphi\left(\sum_{i=1}^n\lambda_i\delta(x_i)\otimes y_i\right)\right| \\
&\leq\left\|\varphi\right\|g_p\left(\sum_{i=1}^n\lambda_i\delta(x_i)\otimes y_i\right) \\
&\leq\left\|\varphi\right\|\left\|(\lambda_1,\ldots,\lambda_n)\right\|_{\ell^n_p}\left\|(y_1,\ldots,y_n)\right\|_{\ell^{n,w}_{p^*}(F)}.
\end{align*}
Since 
\begin{align*}
\left\|(y_1,\ldots,y_n)\right\|_{\ell^{n,w}_{p^*}(F)} &=\displaystyle\sup_{y^{*}\in B_{F^{*}}}\displaystyle\left(\sum_{i=1}^n\left|y^{*}(y_i)\right|^{p^*}\right)^{\frac{1}{p^*}} \\
&=\displaystyle\sup_{y^{*}\in B_{F^{*}}}\displaystyle\left(\sum_{i=1}^n\left|\left\langle y^{*},T_{(\varepsilon,V,W)}(y^{**}_i)\right\rangle\right|^{p^*}\right)^{\frac{1}{p^*}} \\
&=\displaystyle\sup_{y^{*}\in B_{F^{*}}}\displaystyle\left(\sum_{i=1}^n\left|\left\langle \kappa_F(T_{(\varepsilon,V,W)}(y^{**}_i)),y^{*}\right\rangle\right|^{p^*}\right)^{\frac{1}{p^*}} \\
&\leq \left\|\kappa_F\circ T_{(\varepsilon, V,W)}\right\|\sup_{y^{*}\in B_{F^{*}}}\displaystyle\left(\sum_{i=1}^n\left|\left\langle y^{**}_i,y^{*}\right\rangle\right|^{p^*}\right)^{\frac{1}{p^*}} \\
&=\left\|T_{(\varepsilon, V,W)}\right\|\sup_{y^{*}\in B_{F^{*}}}\displaystyle\left(\sum_{i=1}^n\left|\left\langle \kappa_{F^*}(y^*),y^{**}_i\right\rangle\right|^{p^*}\right)^{\frac{1}{p^*}}
\\
&\leq (1+\varepsilon)\displaystyle\sup_{y^{***}\in B_{F^{***}}}\displaystyle\left(\sum_{i=1}^n\left|y^{***}(y^{**}_i)\right|^{p^*}\right)^{\frac{1}{p^*}}\\
&=(1+\varepsilon)\left\|(y^{**}_1,\ldots,y^{**}_n)\right\|_{\ell^{n,w}_{p^*}(F^{**})},
\end{align*}
it follows that 
$$
\left|\sum_{i=1}^n\lambda_i\left\langle y_i^{**},f_\varphi(x_i)\right\rangle\right| \leq\left\|\varphi\right\|\left\|(\lambda_1,\ldots,\lambda_n)\right\|_{\ell^n_p}(1+\varepsilon)\left\|(y^{**}_1,\ldots,y^{**}_n)\right\|_{\ell^{n,w}_{p^*}(F^{**})}. 
$$
By the arbitrariness of $\varepsilon$, we deduce that 
$$
\left|\sum_{i=1}^n\lambda_i\left\langle
y_i^{**},f_\varphi(x_i)\right\rangle\right| \leq\left\|\varphi\right\|\left\|(\lambda_1,\ldots,\lambda_n)\right\|_{\ell^n_p}\left\|(y^{**}_1,\ldots,y^{**}_n)\right\|_{\ell^{n,w}_{p^*}(F^{**})}, 
$$
and this implies that $f_\varphi\in\mathcal{D}^{H^\infty}_p(U,F^*)$ with $d^{\H^\infty}_p(f_\varphi)\leq \left\|\varphi\right\|$.

For any $u=\sum_{i=1}^n \lambda_i\delta(x_i)\otimes y_i\in
\Delta(U)\otimes F$, we get 
$$
\Lambda(f_\varphi)(u) =\sum_{i=1}^n\lambda_i\left\langle f_\varphi(x_i),y_i\right\rangle =\sum_{i=1}^n\lambda_i\varphi(\delta(x_i)\otimes y_i) =\varphi\left(\sum_{i=1}^n\lambda_i\delta(x_i)\otimes y_i\right) =\varphi(u). 
$$
Hence $\Lambda(f_\varphi)=\varphi$ on a dense subspace of $\G^\infty(U)\widehat{\otimes}_{g_p} F$ and, consequently, $\Lambda(f_\varphi)=\varphi$, which shows the last statement of the theorem. Moreover, $d^{\H^\infty}_p(f_\varphi)\leq\left\|\varphi\right\|=\left\|\Lambda(f_\varphi)\right\|$ and the theorem holds.
\end{proof}

In particular, in view of Theorem \ref{messi-10} and taking into account Propositions \ref{ideal summing}, \ref{pies L} and \ref{new-now}, we can identify the space $\H^\infty(U,F^*)$ with the dual space of $\G^\infty(U)\widehat{\otimes}_H F$.

\begin{corollary}\label{theo-dual-classic} 
$\H^\infty(U,F^*)$ is isometrically isomorphic to $(\G^\infty(U)\widehat{\otimes}_H F)^*$, via $\Lambda\colon\H^\infty(U,F^*)\to(\G^\infty(U)\widehat{\otimes}_H F)^*$ given by 
$$
\Lambda(f)(u)=\sum_{i=1}^n\lambda_i\left\langle f(x_i),y_i\right\rangle 
$$
for $f\in\H^\infty(U,F^*)$ and $u=\sum_{i=1}^n\lambda_i\delta(x_i)\otimes y_i\in\Delta(U)\otimes F$. Furthermore, its inverse 
is given by 
$$
\left\langle
\Lambda^{-1}(\varphi)(x),y\right\rangle=\left\langle\varphi,\delta(x)\otimes y\right\rangle 
$$
for $\varphi\in(\G^\infty(U)\widehat{\otimes}_H F)^*$, $x\in U$ and $y\in F$.$\hfill\qed$
\end{corollary}


\begin{remark}
It is known (see \cite[p. 24]{Rya-02}) that if $E$ and $F$ are Banach spaces, then $\mathcal{L}(E,F^*)$ is isometrically isomorphic to $(E\widehat{\otimes}_\pi F)^*$, via $\Phi\colon\mathcal{L}(E,F^*)\to(E\widehat{\otimes}_\pi F)^*$ given by 
$$
\left\langle \Phi(T),\sum_{i=1}^n x_i\otimes y_i\right\rangle=\sum_{i=1}^n\left\langle T(x_i),y_i\right\rangle 
$$
for $T\in\mathcal{L}(E,F^*)$ and $\sum_{i=1}^n x_i\otimes y_i\in E\otimes F$. Notice that the identification $\Lambda$ in Corollary \ref{theo-dual-classic} is justly $\Phi\circ\Phi_0$, where $\Phi_0\colon
f\mapsto T_f$ is the isometric isomorphism from $\H^\infty(U,F^*)$ onto $\mathcal{L}(\G^\infty(U),F)$ given in Theorem \ref{teo1}.
\end{remark}


\section{Pietsch domination for Cohen strongly $p$-summing holomorphic mappings}\label{5}

In \cite{Pie-67}, Pietsch established a domination theorem for $p$-summing linear operators between Banach spaces. In order to present a version of this theorem for Cohen strongly $p$-summing holomorphic mappings on Banach spaces, we first characterize the elements of the dual space of $\Delta(U)\otimes_{g_p} F$.

\begin{theorem}
\label{Pietsch0} 
Let $\varphi\in(\Delta(U)\otimes F)^{\prime }$, $C>0$ and $1<p\leq\infty$. The following conditions are equivalent:

\begin{enumerate}
\item $\left|\varphi(u)\right|\leq Cg_p(u)$ for all $u\in\Delta(U)\otimes F$.
\item For any representation $\sum_{i=1}^n\lambda_i\delta(x_i)\otimes y_i$ of $u\in\Delta(U)\otimes F$, we have 
$$
\sum_{i=1}^n\left|\varphi(\lambda_i\delta(x_i)\otimes y_i)\right|\leq Cg_p(u). 
$$
\item There exists a Borel regular probability measure $\mu$ on $B_{F^{*}}$ such that 
$$
\left|\varphi(\lambda\delta(x)\otimes y)\right|\leq C\left|\lambda\right|\left\|y\right\|_{L_{p^*}(\mu)} 
$$
for all $\lambda\in\mathbb{C}$, $x\in U$ and $y\in F$, where 
$$
\left\|y\right\|_{L_{p^*}(\mu)}=\left(\int_{B_{F^{*}}}\left|y^*(y)\right|^{p^*}\ d\mu(y^*)\right)^{\frac{1}{p^*}}. 
$$
\end{enumerate}
\end{theorem}

\begin{proof}
$(i) \Rightarrow (ii)$: Let $u\in\Delta(U)\otimes F$ and let $\sum_{i=1}^n\lambda_i\delta(x_i)\otimes y_i$ be a representation of $u$. It is elementary that the function $T\colon\mathbb{C}^n\to\mathbb{C}$
defined by 
$$
T(t_1,\ldots,t_n)=\sum_{i=1}^n t_i\varphi(\lambda_i\delta(x_i)\otimes y_i),\qquad \forall (t_1,\ldots,t_n)\in\mathbb{C}^n 
$$
is linear and continuous on $(\mathbb{C}^n,\left\|\cdot\right\|_{\ell_\infty^n})$ with 
$$
\left\|T\right\|=\sum_{i=1}^n\left|\varphi(\lambda_i\delta(x_i)\otimes y_i)\right|. 
$$
For any $(t_1,\ldots,t_n)\in\mathbb{C}^n$ with $\left\|(t_1,\ldots,t_n)\right\|_{\ell_\infty^n}\leq 1$, by (i) we have 
\begin{align*}
\left|T(t_1,\ldots,t_n)\right| &=\left|\sum_{i=1}^nt_i\varphi(\lambda_i\delta(x_i)\otimes y_i)\right| \\
&=\left|\varphi\left(\sum_{i=1}^nt_i\lambda_i\delta(x_i)\otimes y_i\right)\right| \\
&\leq C g_p\left(\sum_{i=1}^nt_i\lambda_i\delta(x_i)\otimes y_i\right) \\
&\leq C\left\|(t_1\lambda_1,\ldots,t_n\lambda_n)\right\|_{\ell^n_p}\left\|(y_1,\ldots,y_n)\right\|_{\ell^{n,w}_{p^*}(F)} \\
&\leq C\left\|(\lambda_1,\ldots,\lambda_n)\right\|_{\ell^n_p}\left\|(y_1,\ldots,y_n)\right\|_{\ell^{n,w}_{p^*}(F)},
\end{align*}
and therefore 
$$
\sum_{i=1}^n\left|\varphi(\lambda_i\delta(x_i)\otimes y_i)\right|\leq C\left\|(\lambda_1,\ldots,\lambda_n)\right\|_{\ell^n_p}\left\|(y_1,\ldots,y_n)\right\|_{\ell^{n,w}_{p^*}(F)}. 
$$
Taking infimum over all the representations of $u$, we deduce that 
$$
\sum_{i=1}^n\left|\varphi(\lambda_i\delta(x_i)\otimes y_i)\right|\leq C g_p(u). 
$$

$(ii)\Rightarrow (iii)$: Let $\mathcal{C}$ be the set of all Borel regular probability measures $\mu $ on $B_{F^*}$. Clearly, it is a convex compact subset of $(C(B_{F^*})^*,w^*)$. Assume first $1<p<\infty$. Let $M$ be set of all functions from $\mathcal{C}$ to $\mathbb{R}$ of the form 
$$
f_{((\lambda _{i})_{i=1}^{n},(x_{i})_{i=1}^{n},(y_{i})_{i=1}^{n})}(\mu)
=\sum_{i=1}^{n}\left\vert \varphi (\lambda _{i}\delta _{U}(x_{i})\otimes y_{i})\right\vert
-\left(\frac{C}{p}\left\Vert(\lambda_{i})_{i=1}^{n}\right\Vert_{\ell _{p}^{n}}^{p}
+\frac{C}{p^*}\sum_{i=1}^{n}\left\Vert y_{i}\right\Vert _{L_{p^{\ast }}(\mu )}^{p^{\ast }}\right) ,
$$
where $n\in \mathbb{N}$, $\lambda _{i}\in \mathbb{C}$, $x_{i}\in U$ and $y_{i}\in F$ for $i=1,\ldots ,n$.

We now check that $M$ satisfies the three conditions of Ky Fan's lemma (see \cite[9.10]{DisJarTon-95}):

\begin{cd}
Each $f_{((\lambda _{i})_{i=1}^{n},(x_{i})_{i=1}^{n},(y_{i})_{i=1}^{n})}\in M$ is convex and lower semicontinuous.
\end{cd}

It suffices to show that $f:=f_{((\lambda _{i})_{i=1}^{n},(x_{i})_{i=1}^{n},(y_{i})_{i=1}^{n})}$ is affine and Lipschitz. Denoting 
$$
A=\sum_{i=1}^{n}\left\vert \varphi (\lambda _{i}\delta _{U}(x_{i})\otimes y_{i})\right\vert-\frac{C}{p}\left\Vert(\lambda_{i})_{i=1}^{n}\right\Vert_{\ell _{p}^{n}}^{p},
$$
we can write
$$
f(\mu)=A-\frac{C}{p^*}\sum_{i=1}^{n}\int_{B_{F^{*}}}\left|y^*(y_i)\right|^{p^*}\ d\mu(y^*)\qquad \left(\mu\in\mathcal{C}\right).
$$
Given $\mu_1,\mu_2\in\mathcal{C}$ and $\alpha\in [0,1]$, we have 
\begin{align*}
f(\alpha\mu_1+(1-\alpha)\mu_2)
&=\left(\alpha A+(1-\alpha)A\right)-\frac{C}{p^*}\sum_{i=1}^{n}\int_{B_{F^{*}}}\left|y^*(y_i)\right|^{p^*}\ d(\alpha\mu_1+(1-\alpha)\mu_2)(y^*)\\
&=\alpha f(\mu_1)+(1-\alpha)f(\mu_2)
\end{align*}
and
\begin{align*}
\left|f(\mu_1)-f(\mu_2)\right|
&=\frac{C}{p^*}\left|\int_{B_{F^{*}}}\left(\sum_{i=1}^{n}\left|y^*(y_i)\right|^{p^*}\right)d(\mu_1-\mu_2)(y^*)\right|\\
&\leq \frac{C}{p^*}\int_{B_{F^{*}}}\left(\sum_{i=1}^{n}\left\|y_i\right\|^{p^*}\right)d|\mu_1-\mu_2|(y^*)\\
&=\frac{C}{p^*}\left\|(y_1,\ldots,y_n)\right\|_{\ell_{p^*}^n}^{p^*}\left\|\mu_1-\mu_2\right\|.
\end{align*}

\begin{cd}
If $g\in \mathrm{co}(M)$, there is an $f_{((\lambda_{i})_{i=1}^{n},(x_{i})_{i=1}^{n},(y_{i})_{i=1}^{n})}\in M$ with $g(\mu)\leq f_{((\lambda _{i})_{i=1}^{n},(x_{i})_{i=1}^{n},(y_{i})_{i=1}^{n})}(\mu
)$ for all $\mu \in \mathcal{C}$. 
\end{cd}

It suffices to show that $M$ is convex. Let $f_1,f_2$ be in $M$ such that
$$
f_1(\mu)=f_{\left(\left(\lambda _{1i}\right) _{i=1}^{s_{1}},\left(x_{1i}\right)_{i=1}^{s_{1}},\left(y_{1i}\right) _{i=1}^{s_{1}}\right)}(\mu)
=\sum_{i=1}^{s_{1}}\left\vert \varphi \left( \lambda _{1i}\delta_{U}(x_{1i})\otimes y_{1i}\right) \right\vert
-\left(\frac{C}{p}\left\Vert \left(\lambda _{1i}\right)_{i=1}^{s_{1}}\right\Vert_{\ell _{p}^{s_{1}}}^{p}+\frac{C}{p^*}\sum_{i=1}^{s_{1}}\left\Vert y_{1i}\right\Vert_{L_{p^*}(\mu )}^{p^*}\right)
$$
and
$$
f_2(\mu)=f_{\left( \left( \lambda _{2i}\right) _{i=1}^{s_{2}},\left( x_{2i}\right)_{i=1}^{s_{2}},\left( y_{2i}\right) _{i=1}^{s_{2}}\right) }(\mu )
=\sum_{i=1}^{s_{2}}\left\vert \varphi \left( \lambda _{2i}\delta_{U}(x_{2i})\otimes y_{2i}\right) \right\vert
-\left(\frac{C}{p}\left\Vert \left( \lambda _{2i}\right) _{i=1}^{s_{2}}\right\Vert_{\ell _{p}^{s_{2}}}^{p}+\frac{C}{p^*}\sum_{i=1}^{s_{2}}\left\Vert y_{2i}\right\Vert_{L_{p^*}(\mu )}^{p^*}\right).
$$
for all $\mu\in\mathcal{C}$. Given $\alpha\in [0,1]$, an easy verification shows that 
$$
\alpha f_{1}(\mu)+\left( 1-\alpha \right) f_{2}(\mu)
=\sum_{i=1}^{n}\left\vert \varphi \left( \lambda _{i}\delta_{U}(x_{i})\otimes y_{i}\right) \right\vert
-\left(\frac{C}{p}\left\Vert \left( \lambda _{i}\right) _{i=1}^{n}\right\Vert _{\ell_{p}^{n}}^{p}+\frac{C}{p^*}\sum_{i=1}^{n}\left\Vert y_{i}\right\Vert _{L_{p^*}(\mu )}^{p^*}\right)
$$
for all $\mu\in\mathcal{C}$, with $n=s_{1}+s_{2}$ and
\begin{align*}
x_{i} &=\left\{\begin{array}{lll}
x_{1i} & \text{if} & 1\leq i\leq s_{1}, \\ 
x_{2(i-s_{1})} & \text{if} & s_{1}+1\leq i\leq n,
\end{array}\right. \\
y_{i} &=\left\{\begin{array}{lll}
\alpha^{1/p^*}y_{1i} & \text{if} & 1\leq i\leq s_{1}, \\ 
(1-\alpha)^{1/p^*}y_{2(i-s_{1})}& \text{if} & s_{1}+1\leq i\leq n,
\end{array}\right. \\
\lambda_{i} &=\left\{\begin{array}{lll}
\alpha^{1/p}\lambda_{1i} & \text{if} & 1\leq i\leq s_{1},\\ 
(1-\alpha)^{1/p}\lambda_{2(i-s_{1})} & \text{if} & s_{1}+1\leq i\leq n.
\end{array}\right. 
\end{align*}

\begin{cd}
There exists an $r\in \mathbb{R}$ such that each $f_{((\lambda_{i})_{i=1}^{n},(x_{i})_{i=1}^{n},(y_{i})_{i=1}^{n})}\in M$ has a value less or equal than $r$. 
\end{cd}

Let us show that $r=0$ verifies this condition. Let $f_{((\lambda_{i})_{i=1}^{n},(x_{i})_{i=1}^{n},(y_{i})_{i=1}^{n})}\in M$. There exists $y_{0}^*\in B_{F^*}$ such that
$$
\sup_{y^*\in B_{F^*}}\left(\sum_{i=1}^{n}\left\vert y^*(y_i)\right\vert ^{p^*}\right)^{\frac{1}{p^*}}
=\left(\sum_{i=1}^{n}\left\vert y_{0}^*(y_{i})\right\vert ^{p^*}\right)^{\frac{1}{p^*}}.
$$
Let $\delta _{y_{0}^*}$ be the Dirac's measure on $B_{F^*}$ supported by $y_{0}^*$. Taking 
$$
\alpha =\left\Vert \left( \lambda _{i}\right) _{i=1}^{n}\right\Vert _{\ell_{p}^{n}},\; \beta=\left(\sum_{i=1}^{n}\left\vert y_{0}^*(y_i)\right\vert ^{p^*}\right)^{\frac{1}{p^*}}
$$
in the identity (see \cite[p. 48]{Jar-81}): 
$$
\alpha \beta=\min_{\epsilon >0}\left\{\frac{1}{p}\left(\frac{\alpha }{\epsilon }\right)^{p}+\frac{1}{p^*}\left(\epsilon\beta\right)^{p^*}\right\}
\qquad \left(\alpha,\beta\in\mathbb{R}^{+}\right),
$$
we obtain that 
\begin{align*}
f\left(\delta_{y^*_0}\right)
&=\sum_{i=1}^{n}\left\vert \varphi \left( \lambda _{i}\delta_{U}(x_{i})\otimes y_{i}\right) \right\vert
-\left(\frac{C}{p}\left\Vert\left(\lambda_{i}\right)_{i=1}^{n}\right\Vert_{\ell_{p}^{n}}^{p}
+\frac{C}{p^*}\sum_{i=1}^{n}\left\Vert y_{i}\right\Vert_{L_{p^*}(\delta_{y_{0}^*})}^{p^*}\right) \\ 
&=\sum_{i=1}^{n}\left\vert\varphi\left(\lambda_{i}\delta_{U}(x_{i})\otimes y_{i}\right)\right\vert
-\left(\frac{C}{p}\left\Vert \left( \lambda _{i}\right) _{i=1}^{n}\right\Vert _{\ell_{p}^{n}}^{p}
+\frac{C}{p^*}\sum_{i=1}^{n}\left\vert y_{0}^*(y_{i})\right\vert ^{p^*}\right) \\ 
&\leq \sum_{i=1}^{n}\left\vert \varphi \left(\lambda_{i}\delta_{U}(x_{i})\otimes y_{i}\right)\right\vert
-C\left\Vert \left(\lambda_{i}\right)_{i=1}^{n}\right\Vert_{\ell_{p}^{n}}\left(\sum_{i=1}^{n}\left\vert y_{0}^*(y_{i})\right\vert ^{p^*}\right)^{\frac{1}{p^*}}\\ 
&\leq  0
\end{align*}
by (ii). By Ky Fan's lemma, there exists a $\mu \in \mathcal{C}$ such that $f(\mu )\leq 0$ for all $f\in M$. In particular, we have 
$$
f_{(t\lambda ,x,t^{-1}y)}(\mu )=\left\vert \varphi (t\lambda \delta
_{U}(x)\otimes t^{-1}y)\right\vert -\frac{C}{p}t^{p}\left\vert \lambda
\right\vert ^{p}-\frac{C}{p^*}t^{-p^*}\left\Vert y\right\Vert
_{L_{p^*}(\mu )}^{p^*}\leq 0
$$
for all $t\in \mathbb{R}^{+}$, $\lambda \in \mathbb{C}$, $x\in U$ and $y\in F
$. It follows that 
$$
\left\vert \varphi (\lambda \delta _{U}(x)\otimes y)\right\vert \leq C\left( 
\frac{t^{p}\left\vert \lambda \right\vert ^{p}}{p}+\frac{t^{-p^{\ast
}}\left\Vert y\right\Vert _{L_{p^*}(\mu )}^{p^*}}{p^*}\right) ,
$$
and, applying again the aforementioned identity, we conclude that  
$$\left\vert \varphi (\lambda \delta _{U}(x)\otimes y)\right\vert \leq C\left\vert \lambda \right\vert \left\Vert y\right\Vert _{L_{p^*}(\mu)}.
$$
The case $p=\infty$ is similarly proved but without applying the cited identity and taking $C/p=0$ and $p^*=1$.

$(iii) \Rightarrow (i)$: Let $u\in\Delta(U)\otimes F$ and let $%
\sum_{i=1}^n\lambda_i\delta(x_i)\otimes y_i$ be a representation of $u$.
Using (iii) and the H\"older inequality, we obtain 
\begin{align*}
\left|\varphi(u)\right|
&\leq\sum_{i=1}^n\left|\varphi\left(\lambda_i\delta(x_i)\otimes
y_i\right)\right| \\
&\leq C\sum_{i=1}^n\left|\lambda_i\right|\left\|y_i\right\|_{L_{p^*}(\mu)} \\
&\leq
C\left\|(\lambda_1,\ldots,\lambda_n)\right\|_{\ell^n_p}\left(\sum_{i=1}^n%
\left\|y_i\right\|^{p^*}_{L_{p^*}(\mu)}\right)^{\frac{1}{p^*}} \\
&=C\left\|(\lambda_1,\ldots,\lambda_n)\right\|_{\ell^n_p}\left(%
\int_{B_{F^{*}}}\sum_{i=1}^n\left|y^*(y_i)\right|^{p^*}\ d\mu(y^*)\right)^{%
\frac{1}{p^*}} \\
&\leq
C\left\|(\lambda_1,\ldots,\lambda_n)\right\|_{\ell^n_p}\left(\sup_{y^*\in
B_{F^*}}\sum_{i=1}^n\left|y^*(y_i)\right|^{p^*}\right)^{\frac{1}{p^*}} \\
&=C\left\|(\lambda_1,\ldots,\lambda_n)\right\|_{\ell^n_p}\left\|(y_1,%
\ldots,y_n)\right\|_{\ell^{n,w}_{p^*}(F)},
\end{align*}
and taking infimum over all the representations of $u$, we conclude that $%
\left|\varphi(u)\right|\leq Cg_p(u)$.
\end{proof}

We are now ready to present the announced result. Compare to \cite[Theorem 2.3.1]{Coh-73}.

\begin{theorem}\label{Pietsch}(Pietsch domination). 
Let $1<p\leq\infty$ and $f\in\H^\infty(U,F)$. The following conditions are equivalent:
\begin{enumerate}
\item $f$ is Cohen strongly $p$-summing holomorphic.
\item For any $\sum_{i=1}^n\lambda_i\delta(x_i)\otimes y^*_i\in\Delta(U)\otimes F^*$, we have 
$$
\left|\sum_{i=1}^n\lambda_i\left\langle y^*_i,f(x_i)\right\rangle\right|\leq d^{\H^\infty}_p(f)\left\|(\lambda_1,\ldots,\lambda_n)\right\|_{\ell^n_p}\left\|(y^*_1,\ldots,y^*_n)\right\|_{\ell^{n,w}_{p^*}(F^*)}. 
$$
\item There is a constant $C>0$ and a Borel regular probability measure $\mu $ on $B_{F^{**}}$ such that 
$$
\left\vert \left\langle y^*,f(x)\right\rangle \right\vert \leq C\left\Vert y^*\right\Vert _{L_{p^*}(\mu )}
$$
for all $x\in U$ and $y^*\in F^*$, where 
$$
\left\Vert y^*\right\Vert _{L_{p^*}(\mu )}=\left(\int_{B_{F^{**}}}\left\vert y^{**}(y^*)\right\vert^{p^*}d\mu (y^{**})\right) ^{\frac{1}{p^*}}.
$$
\end{enumerate}
In this case, $d_p^{\H^\infty}(f)$ is the minimum of all constants $C>0$ satisfying the preceding inequality.
\end{theorem}

\begin{proof}
$(i) \Rightarrow (ii)$ is immediate from Definition \ref{def-Csps}.

$(ii) \Rightarrow (iii)$: Clearly, $\kappa_F\circ f\in\H^\infty(U,F^{**})$. Appealing to Corollary \ref{linearization}, consider its associate linear functional $\Lambda_0(\kappa_F\circ f)\colon\Delta(U)\otimes F^*\to\mathbb{C}$. Given $u=\sum_{i=1}^n\lambda_i\delta(x_i)\otimes y^*_i\in\Delta(U)\otimes F^*$, we have 
\begin{align*}
\left|\Lambda_0(\kappa_F\circ
f)(u)\right|&=\left|\sum_{i=1}^n\lambda_i\left\langle (\kappa_F\circ
f)(x_i),y^*_i\right\rangle\right| \\
&=\left|\sum_{i=1}^n\lambda_i\left\langle y^*_i,f(x_i)\right\rangle\right| \\
&\leq d_p^{\H^\infty}(f)\left\|(\lambda_1,\ldots,\lambda_n)\right\|_{\ell^n_p}\left\|(y^*_1,\ldots,y^*_n)\right\|_{\ell^{n,w}_{p^*}(F^*)}
\end{align*}
by (ii). Since it holds for each representation of $u$, we deduce that 
$$
\left|\Lambda_0(\kappa_F\circ f)( u)\right|\leq d_p^{\H^\infty}(f)g_p( u). 
$$
By Theorem \ref{Pietsch0}, there exists a Borel regular probability measure $\mu$ on $B_{F^{**}}$ such that 
\begin{align*}
\left|\left\langle y^*,f(x)\right\rangle\right|
&=\left|\Lambda_0(\kappa_F\circ f)(\delta(x)\otimes y^*)\right| \\
&\leq d_p^{\H^\infty}(f)\left(\int_{B_{F^{**}}}\left|y^{**}(y^*)\right|^{p^*}\ d\mu(y^{**})\right)^{\frac{1}{p^*}}
\end{align*}
for all $x\in U$ and $y^*\in F^*$. Moreover, $d_p^{\H^\infty}(f)$ belongs to the set of all constants $C>0$ satisfying the inequality in (iii).

$(iii)\Rightarrow(i)$: Given $x\in U$ and $y^*\in F^*$, we have 
$$
\left|\Lambda_0(\kappa_F\circ f)(\delta(x)\otimes y^*)\right|=\left|\left\langle y^*,f(x)\right\rangle\right|\leq \left\Vert y^*\right\Vert _{L_{p^*}(\mu )}
$$
by applying (iii). Now, Theorem \ref{Pietsch0} guarantees that for any
representation $\sum_{i=1}^n\lambda_i\delta(x_i)\otimes y^*_i$ of $u\in\Delta(U)\otimes F^*$, we have 
\begin{align*}
\sum_{i=1}^n\left|\lambda_i\right|\left|\left\langle y^*_i,f(x_i)\right\rangle\right|
&=\sum_{i=1}^n\left|\lambda_i\right|\left|\left\langle(\kappa_F\circ f)(x_i),y^*_i\right\rangle\right| \\
&=\sum_{i=1}^n\left|\Lambda_0(\kappa_F\circ f)(\lambda_i\delta(x_i)\otimes y^*_i)\right| \\
&\leq Cg_p( u) \\
&\leq
C\left\|(\lambda_1,\ldots,\lambda_n)\right\|_{\ell^n_p}\left\|(y^*_1,\ldots,y^*_n)\right\|_{\ell^{n,w}_{p^*}(F^*)}.
\end{align*}
Hence $f\in\mathcal{D}^{\H^\infty}_p(U,F)$ with $d^{\H^\infty}_p(f)\leq C$. This also shows the last assertion of the statement.
\end{proof}

We now study the relationship between a Cohen strongly $p$-summing holomorphic mapping from $U$ to $F$ and its associate linearization from $\G^\infty(U)$ to $F$.

\begin{theorem}\label{summing}
Let $1<p\leq\infty$ and $f\in\H^\infty(U,F)$. The following conditions are equivalent:
\begin{enumerate}
\item $f\colon U\to F$ is Cohen strongly $p$-summing holomorphic. 
\item $T_f\colon\G^\infty(U)\to F$ is strongly $p$-summing.
\end{enumerate}
In this case, $d_{p}(T_{f})=d_{p}^{\H^{\infty}}(f)$. Furthermore, the mapping $f\mapsto T_f$ is an isometric isomorphism from $(\mathcal{D}_p^{\H^\infty}(U,F),d_p^{\H^{\infty}})$ onto $(\mathcal{D}_p(\G^\infty(U),F),d_p)$.
\end{theorem}

\begin{proof}
$(i) \Rightarrow (ii)$: Assume that $f\in \mathcal{D}_{p}^{\H^{\infty }}(U,F)$. By Theorem \ref{Pietsch}, there exists a constant $C>0$ and a Borel regular probability measure $\mu $ on $B_{F^{**}}$ such that 
$$
\left\vert \left\langle y^*,f(x)\right\rangle \right\vert \leq C\left\Vert y^*\right\Vert _{L_{p^*}(\mu )}
$$
for all $x\in U$ and $y^*\in F^*$.

Let $y^*\in F^*$ and $\gamma\in\G^{\infty }(U)$. By Theorem \ref{teo1}, given $\varepsilon>0$, we can take a representation $\sum_{i=1}^{\infty}\lambda_{i}\delta(x_i)$ of $\gamma$ such that $\sum_{i=1}^{\infty}\left|\lambda_i\right|\leq\left\|\gamma\right\|+\varepsilon$. We have 
\begin{align*}
\left\vert \left\langle y^*,T_{f}(\gamma )\right\rangle \right\vert 
&=\left\vert \left\langle y^*,\sum_{i=1}^{\infty}\lambda _{i}T_{f}(\delta_{U}(x_{i}))\right\rangle \right\vert  \\
& =\left\vert \left\langle y^*,\sum_{i=1}^{\infty}\lambda_{i}f(x_{i})\right\rangle \right\vert  \\
& \leq \sum_{i=1}^{\infty}\left\vert \lambda _{i}\right\vert \left\vert\left\langle y^*,f(x_{i})\right\rangle \right\vert  \\
& \leq C\left\Vert y^*\right\Vert _{L_{p^*}(\mu)}\sum_{i=1}^{\infty}\left\vert \lambda _{i}\right\vert \\
&\leq C\left\Vert y^*\right\Vert _{L_{p^*}(\mu)}\left(\left\|\gamma\right\|+\varepsilon\right).
\end{align*}
As $\varepsilon$ was arbitrary, it follows that  
$$
\left\vert \left\langle y^*,T_{f}(\gamma )\right\rangle \right\vert\leq C\left\Vert y^*\right\Vert _{L_{p^*}(\mu )}\left\Vert\gamma \right\Vert .
$$
Taking infimum over all such constants $C$, we have 
$$
\left\vert \left\langle y^*,T_{f}(\gamma )\right\rangle \right\vert\leq d_{p}^{\H^{\infty }}(f)\left\Vert y^*\right\Vert_{L_{p^*}(\mu )}\left\Vert \gamma \right\Vert 
$$
by Theorem \ref{Pietsch}. It follows that 
$$
\sup \left\{ \left\vert \left\langle y^*,T_{f}(\gamma )\right\rangle\right\vert \colon y^*\in F^*,\ \left\Vert y^*\right\Vert_{L_{p^*}(\mu )}\leq 1\right\} \leq d_{p}^{\H^{\infty}}(f)\left\Vert \gamma \right\Vert 
$$
for all $\gamma \in \G^{\infty }(U)$. Therefore $T_{f}\in \mathcal{D}_{p}(\G^{\infty }(U),F)$ with $d_{p}(T_{f})\leq d_{p}^{\H^{\infty }}(f)$ by Pietsch's domination theorem for strongly $p$-summing
operators \cite[Theorem 2.3.1]{Coh-73}.

$(ii)\Rightarrow(i)$: Assume that $T_{f}\in \mathcal{D}_{p}(\G^{\infty }(U),F)$. Given $x\in U$ and $y^*\in F^*$, we have 
\begin{align*}
\left\vert \left\langle y^*,f(x)\right\rangle \right\vert 
& =\left\vert \left\langle y^*,T_{f}(\delta _{U}(x)\right\rangle \right\vert  \\
& \leq d_{p}(T_{f})\left\Vert y^*\right\Vert _{L_{p^*(\mu )}}\left\Vert \delta _{U}(x)\right\Vert  \\
& =d_{p}(T_{f})\left\Vert y^*\right\Vert _{L_{p^*(\mu )}}
\end{align*}%
by \cite[Theorem 2.3.1]{Coh-73} for some Borel regular probability measure $\mu $ on $B_{F^{**}}$. It follows that $f\in \mathcal{D}_{p}^{\H^{\infty }}(U,F)$ with $d_{p}^{\H^{\infty }}(f)\leq d_{p}(T_{f})$
by Theorem \ref{Pietsch}.

Since $d_{p}(T_{f})=d_{p}^{\H^{\infty }}(f)$ for all $f\in \mathcal{D}_{p}^{\H^{\infty }}(U,F)$, in order to prove the last assertion of the statement, it suffices to show that the mapping $f\mapsto T_{f}$ from $\mathcal{D}_{p}^{\H^{\infty }}(U,F)$ to $\mathcal{D}_{p}(\mathcal{G}^{\infty }(U),F)$ is surjective. Indeed, take $T\in \mathcal{D}_{p}(\G^{\infty }(U),F)$ and then $T=T_{f}$ for some $f\in \H^{\infty }(U,F)$ by Theorem \ref{teo1}. Hence $T_{f}\in \mathcal{D}_{p}(\G^{\infty }(U),F)$, and thus $f\in \mathcal{D}_{p}^{\H^{\infty }}(U,F)$ by the above proof.
\end{proof}

The equivalence $(i)\Leftrightarrow (iii)$ of Theorem \ref{Pietsch} admits the following  reformulation.

\begin{corollary}
Let $1<p\leq\infty$ and $f\in\H^\infty(U,F)$. The following conditions are equivalent:
\begin{enumerate}
\item $f\colon U\to F$ is Cohen strongly $p$-summing holomorphic. 
\item There exists a complex Banach space $G$ and an operator $S\in\D_p(G,F)$ such that 
$$
\left|\left\langle y^*,f(x)\right\rangle\right|\leq \left\|S^*(y^*)\right\|\qquad (x\in U,\; y^*\in F^*).
$$
\end{enumerate}
In this case, $d_p^{\H^\infty}(f)$ is the infimum of all $d_p(S)$ with $S$ satisfying (ii), and this infimum is attained at $T_f$ (\textit{Mujica's linearization of $f$}).
\end{corollary}

\begin{proof}
$(i) \Rightarrow (ii)$: If $f\in\D_p^{\H^\infty}(U,F)$, then $T_f\in\D_p(\G^\infty(U),F)$ with $d_{p}^{\H^{\infty}}(f)=d_{p}(T_{f})$ by Theorem \ref{summing}. From Theorem \ref{teo1}, we infer that 
\begin{align*}
\left|\left\langle y^*,f(x)\right\rangle\right|&=\left|\left\langle y^*,T_{f}(\delta _{U}(x)\right\rangle\right|\\
                                               &=\left|\left\langle (T_{f})^*(y^*),\delta _{U}(x)\right\rangle\right|\\
                                               &\leq \left\|(T_{f})^*(y^*)\right\| 
\end{align*}
for all $x\in U$ and $y^*\in F^*$.

$(ii)\Rightarrow(i)$: Assume that (ii) holds. Then $S^*\in\Pi_{p^*}(F^*,G^*)$ with $\pi_{p^*}(S^*)=d_p(S)$ by \cite[Theorem 2.2.2]{Coh-73}. By Pietsch domination theorem for $p$-summing linear operators (see \cite[Theorem 2.12]{DisJarTon-95}), there is a Borel regular probability measure $\mu$ on $B_{F^{**}}$ such that 
$$
\left\|S^*(y^*)\right\|\leq \pi_{p^*}(S^*)\left\|y^*\right\|_{L_{p^*}(\mu )}
$$
for all $y^*\in F^*$. It follows that 
$$
\left|\left\langle y^*,f(x)\right\rangle\right|\leq \left\|S^*(y^*)\right\|\leq \pi_{p^*}(S^*)\left\|y^*\right\|_{L_{p^*}(\mu )}
$$
for all $x\in U$ and $y^*\in F^*$. Hence $f\in\D_p^{\H^\infty}(U,F)$ with $d_p^{\H^\infty}(f)\leq \pi_{p^*}(S^*)=d_p(S)$ by Theorem \ref{Pietsch}. 
\end{proof}

As a consequence of Theorem \ref{summing}, an application of \cite[Theorem 3.2]{ArBoPelRu10} shows that the Banach ideal $\mathcal{D}_{p}^{\H^{\infty }}$ is generated by composition with the Banach operator ideal $\mathcal{D}_{p}$, but we prefer to give here a proof to complete the information.

\begin{corollary}\label{messi-3}
Let $1<p\leq\infty$ and $f\in\H^\infty(U,F)$. The following conditions are equivalent:
\begin{enumerate}
\item $f\colon U\to F$ is Cohen strongly $p$-summing holomorphic. 
\item $f=T\circ g$ for some complex Banach space $G$, $g\in\H^{\infty}(U,G)$ and $T\in\D_{p}(G,F)$. 
\end{enumerate}
In this case, $d_p^{\H^\infty}(f)=\inf\{d_p(T)\left\|g\right\|_\infty\}$, where the infimum is taken over all factorizations of $f$ as in (ii), and this infimum is attained at $T_f\circ g_U$ (\textit{ Mujica's factorization of $f$}). 
\end{corollary}

\begin{proof}
$(i)\Rightarrow (ii)$: If $f\in\D_p^{\H^\infty}(U,F)$, we have $f=T_f\circ g_U$, where $\G^\infty(U)$ is a complex Banach space, $T_f\in\D_p(\G^\infty(U),F)$ and $g_U\in\H^\infty(U,\G^\infty(U))$ by Theorems \ref{teo1} and \ref{summing}. Moreover,  
$$
\inf\left\{d_p(T)\left\|g\right\|_\infty\right\}\leq d_p(T_f)\left\|g_U\right\|_\infty=d_p^{\H^\infty}(f).
$$

$(ii)\Rightarrow (i)$: Assume $f=T\circ g$ with $G$, $g$ and $T$ being as in (ii). Since $g=T_g\circ g_U$ by Theorem \ref{teo1}, it follows that $f=T\circ T_g\circ g_U$ which implies that $T_f=T\circ T_g$, and thus $T_f\in\D_p(\G^\infty(U),F)$ by the ideal property of $\D_p$. By Theorem \ref{summing}, we obtain that $f\in\D_p^{\H^\infty}(U,F)$ with  
$$
d_p^{\H^\infty}(f)=d_p(T_f)=d_p(T\circ T_g)\leq k_p(T)\left\|T_g\right\|=d_p(T)\left\|g\right\|_\infty, 
$$
and so $d_p^{\H^\infty}(f)\leq\inf\{d_p(T)\left\|g\right\|_\infty\}$ by taking the infimum over all factorizations of $f$. 
\end{proof}

When $F$ is reflexive, every $f\in\D_2^{\H^\infty}(U,F)$ factors through a Hilbert space as we see below.

\begin{corollary}\label{messi-3-3}
Let $F$ be a reflexive complex Banach space. If $f\in\D^{H^\infty}_2(U,F)$, then there exist a Hilbert space $H$, an operator $T\in\D_2(H,F)$ and a mapping $g\in\H^{\infty}(U,H)$ such that $f=T\circ g$. 
\end{corollary}

\begin{proof}
Assume that $f\in\D^{H^\infty}_2(U,F)$. By Theorem \ref{summing}, $T_f\in\D_2(\G^{\infty}(U),F)$. 
Hence $(T_f)^*\in\Pi_2(F^*,\G^{\infty}(U)^*)$ by \cite[Theorem 2.2.2]{Coh-73}. 
By \cite[Corollary 2.16 and Examples 2.9 (b)]{DisJarTon-95}, there exist a Hilbert space $H$ and operators $T_1\in\Pi_2(F^*,H)$ and $T_2\in\L(H,\G^{\infty}(U)^*)$ 
such that $(T_f)^*=T_2\circ T_1$. 

On the one hand, we have $(T_f)^{**}=(T_1)^*\circ (T_2)^*$, where $(T_1)^*\in\D_2(H,F^{**})$ 
by \cite[Theorem 2.2.2]{Coh-73}. On the other hand, we have $(T_f)^{**}\circ\kappa_{\G^{\infty}(U)}=\kappa_F\circ T_f$ with $\kappa_F$ being bijective (since $F$ is reflexive). Consequently, we obtain $f=T\circ g$, where $T=(\kappa_F)^{-1}\circ (T_1)^*\in\D_2(H,F)$ and $g=(T_2)^*\circ\kappa_{\G^{\infty}(U)}\circ g_U\in \H^{\infty}(U,H)$.
\end{proof}

Applying Theorem \ref{summing} and \cite[Theorem 2.4.1]{Coh-73}, we get useful inclusion relations.

\begin{corollary}\label{new1}
Let $1<p_1\leq p_2\leq\infty$. If $f\in\D^{\H^\infty}_{p_2}(U,F)$, then $f\in\D^{\H^\infty}_{p_1}(U,F)$ and $d_{p_1}^{\H^\infty}(f)\leq d_{p_2}^{\H^\infty}(f)$.$\hfill\qed$
\end{corollary}

These inclusion relations can become coincidence relations when $F^*$ has cotype 2 (see \cite[pp. 217--221]{DisJarTon-95} for definitions and results on this class of spaces). Compare to \cite[Corollary 11.16]{DisJarTon-95}.  

\begin{corollary}\label{new2}
Let $2<p\leq\infty$. If $F^*$ has cotype 2, then $\D_p^{\H^\infty}(U,F)=\D_2^{\H^\infty}(U,F)$ and $d_p^{\H^\infty}(f)=d_2^{\H^\infty}(f)$ for all $f\in\D_p^{\H^\infty}(U,F)$. 
\end{corollary}

\begin{proof}
By Corollary \ref{new1}, we have $\D_p^{\H^\infty}(U,F)\subseteq\D_2^{\H^\infty}(U,F)$ with $d_2^{\H^\infty}(f)\leq d_p^{\H^\infty}(f)$ for all $f\in\D_p^{\H^\infty}(U,F)$. For the converse, let $f\in\D_2^{\H^\infty}(U,F)$. Then $T_{f}\in\D_2(\G^{\infty}(U),F)$ with $d_2(T_{f})=d_2^{\H^{\infty}}(f)$ by Theorem \ref{summing}. Hence $(T_{f})^*\in\Pi_2(F^*,\G^{\infty}(U)^*)$ with $\pi_2((T_{f})^*)=d_2(T_f)$ by \cite[Theorem 2.2.2]{Coh-73}. Then, by \cite[Corollary 11.16]{DisJarTon-95}, $(T_{f})^*\in\Pi_1(F^*,\G^{\infty}(U)^*)$ with $\pi_1((T_{f})^*)=\pi_2((T_f)^*)$. Hence $(T_{f})^*\in\Pi_{p^*}(F^*,\G^{\infty}(U)^*)$ with $\pi_{p^*}((T_{f})^*)\leq\pi_1((T_f)^*)$ by \cite[Theorem 2.8]{DisJarTon-95}. Then, by \cite[Theorem 2.2.2]{Coh-73}, $T_{f}\in\D_p(\G^{\infty}(U),F)$ with $d_p(T_{f})=\pi_{p^*}((T_f)^*)$. Finally, $f\in\D_p^{\H^\infty}(U,F)$ with $d_{p}^{\H^{\infty}}(f)=d_{p}(T_{f})$ by Theorem \ref{summing}, and therefore $d_p^{\H^\infty}(f)\leq d_2^{\H^\infty}(f)$.
\end{proof}

Given $f\in\H^\infty(U,F)$, the \textit{transpose} of $f$ is the mapping $f^t\colon F^*\to\H^\infty(U)$ defined by 
$$
f^t(y^*)=y^*\circ f\qquad (y^*\in F^*). 
$$
It is known (see \cite[Proposition 1.6]{JimRuiSep-22}) that $f^t\in\mathcal{L}(F^*,\H^\infty(U))$ with $||f^t||=\left\|f\right\|_\infty$. Furthermore, $f^t=J^{-1}_U\circ (T_f)^*$ with $J_U\colon\H%
^\infty(U)\to\G^\infty(U)^*$ being the identification established in Theorem \ref{teo1}.

The next result establishes the relation of a Cohen strongly $p$-summing holomorphic mapping $f\colon U\to F$ and its transpose $f^t\colon F^*\to\H^\infty(U)$. Compare to \cite[Theorem 2.2.2]{Coh-73}.

\begin{theorem}\label{messi-2}
Let $1<p\leq\infty$ and $f\in\H^\infty(U,F)$. Then $f\in\mathcal{D}^{\H_\infty}_p(U,F^*)$ if and only if $f^t\in\Pi_{p^*}(F^*,\H^\infty(U))$. In this case, $d^{\H^\infty}_p(f)=\pi_{p^*}(f^t)$.
\end{theorem}

\begin{proof}
Applying Theorem \ref{summing}, \cite[Theorem 2.2.2]{Coh-73} and \cite[2.4 and 2.5]{DisJarTon-95}, respectively, we have 
\begin{align*}
f\in\mathcal{D}^{\H_\infty}_p(U,F^*)&\Leftrightarrow T_f\in\mathcal{D}_p(\G^\infty(U),F) \\
&\Leftrightarrow (T_f)^*\in\Pi_{p^*}(F^*,\G^\infty(U)^*) \\
&\Leftrightarrow f^t=J^{-1}_U\circ(T_f)^*\in\Pi_{p^*}(F^*,\H^\infty(U)).
\end{align*}
In this case, $d^{\H^\infty}_p(f)=d_p(T_f)=\pi_{p^*}((T_f)^*)=\pi_{p^*}(J_U\circ f^t)=\pi_{p^*}(f^t)$.
\end{proof}

The study of holomorphic mappings with relatively (weakly) compact range was initiated by Mujica \cite{Muj-91} and followed in \cite{JimRuiSep-22}. 

\begin{corollary}
Let $1<p\leq\infty$. 
\begin{enumerate}
\item Every Cohen strongly $p$-summing holomorphic mapping $f\colon U\to F$ has relatively weakly compact range. 
\item If $F$ is reflexive, then every Cohen strongly $p$-summing holomorphic mapping $f\colon U\to F$ has relatively compact range.
\end{enumerate}
\end{corollary}

\begin{proof}
If $f\in\mathcal{D}^{\H_\infty}_p(U,F^*)$, then $f^t\in\Pi_{p^*}(F^*,\H^\infty(U))$ by Theorem \ref{messi-2}. Hence the linear operator $f^t$ is weakly compact and completely continuous by \cite[2.17]{DisJarTon-95}. Since $f^t$ is weakly compact, this means that $f$ has relatively weakly compact range by \cite[Theorem 2.7]{JimRuiSep-22}. Since $f^t$ is completely continuous and $F^*$ is reflexive, it is known that $f^t$ is compact and, equivalently, $f$ has relatively compact range by \cite[Theorem 2.2]{JimRuiSep-22}.
\end{proof}


\section{Pietsch factorization for Cohen strongly $p$-summing holomorphic mappings}\label{6}

We devote this section to the analogue of Pietsch factorization theorem for $p$-summing linear operators \cite[Theorem 2.13]{DisJarTon-95} for the class of Cohen strongly $p$-summing holomorphic mappings. Recall that, for every Banach space $F$, we have the canonical isometric injections $\kappa_F\colon F\to F^{**}$ and $\iota_F\colon F\to C\left(B_{F^*}\right)$ defined, respectively, by 
\begin{align*}
\left\langle \kappa_F(y),y^*\right\rangle=y^*(y)\qquad \left(y\in F,\; y^*\in F^*\right),\\
\left\langle \iota_F(y),y^*\right\rangle =y^*(y)\qquad \left(y\in F,\; y^*\in B_{F^*}\right).
\end{align*}
Moreover, if $\mu$ is a regular Borel measure on $(B_{F^{**}},w^*)$, $j_{p}$ denotes the canonical map from $C\left(B_{F^*}\right)$ to $L_{p}\left(\mu\right)$.

\begin{theorem}\label{PFactorization}(Pietsch factorization). 
Let $1<p\leq \infty $ and $f\in \H^{\infty }(U,F)$. The following conditions are equivalent:
\begin{enumerate}
\item $f\colon U\rightarrow F$ is Cohen strongly $p$-summing holomorphic.
\item There exist a regular Borel probability measure $\mu$ on $(B_{F^{**}},w^*)$, a closed subspace $S_{p^*}$ of $L_{p^*}(\mu)$ and a bounded holomorphic mapping $g\colon U\rightarrow
(S_{p^*})^*$ such that the following diagram commutes:
$$			
\xymatrix{(S_{p^*})^* \ar[rr]^{(j_{p^*})^* } &&\left(\iota_{F^*}\left(F^*\right) \right) ^* \ar[d]^{(\iota_{F^*})^*} \\
				U\ar[u]^g \ar[r]^f & F \ar[r]^{\kappa_{F}} & F^{**}}
$$
\end{enumerate}
In this case, $d_p^{\H^\infty}(f)=\left\|g\right\|_\infty$.
\end{theorem}

\begin{proof}
$(i)\Rightarrow(ii)$: Let $f\in\mathcal{D}^{\H_\infty}_p(U,F)$. Then $f^t\in\Pi_{p^*}(F^*,\H^\infty(U))$ by Theorem \ref{messi-2}. By \cite[Theorem 2.13]{DisJarTon-95}, there exist a regular Borel probability measure $\mu$ on $(B_{F^{**}},w^*)$, a subspace $S_{p^*}:=\overline{j_{p^*}\left( i_{F^*}\left(F^*\right)\right)}$ of $L_{p^*}(\mu)$, and an operator $T\in\L(S_{p^*},\H^{\infty}(U))$ with $\left\|T\right\|=||f^t||$ such that the following diagram commutes:

\begin{equation*}
\xymatrix{
	 	\iota_{F^*}(F^*) \ar[rr]^{j_{p^*}} & & S_{p^*}\ar[d]^{T} \\
	 	F^*\ar[u]^{\iota_{F^*}} \ar[rr]^{f^t} &  & \H^{\infty}(U)}
\end{equation*}
Dualizing, we obtain
$$
\begin{array}{ccc}
U & \overset{f}{\rightarrow } & F \\ 
\delta_{U}\downarrow  &  & \downarrow \kappa_{F} \\ 
\H^{\infty }(U)^* & \overset{(f^t)^{*}}{\rightarrow } & F^{**} \\ 
T^*\downarrow  &  & (\iota_{F^*})^*\uparrow  \\ 
(S_{p^*})^* & \overset{(j_{p^*})^*}{\rightarrow } & (\iota_{F^*}(F^*))^*
\end{array}
$$
Define $g:=T^*\circ g_U$. Clearly, $g\in\H^\infty(U,(S_{p^*})^*)$ with $\left\|g\right\|_\infty\leq\left\|T\right\|$, and thus 
$$
\left\|g\right\|_\infty\leq ||f^t||=\left\|f\right\|_\infty\leq d_p^{\H^\infty}(f).
$$
Moreover, since $f^t=T\circ j_{p^*}\circ \iota_{F^*}$, we have
$$
\kappa_F\circ f=(f^t)^*\circ g_U=(\iota_{F^*})^*\circ (j_{p^*})^*\circ T^*\circ g_U=(\iota_{F^*})^*\circ (j_{p^*})^*\circ g.
$$

$(ii)\Rightarrow(i)$: Since $\kappa_{F}\circ f=(\iota_{F^*})^*\circ (j_{p^*})^*\circ g$, it follows that $f^{t}\circ(\kappa_F)^*=((\iota_{F^*})^*\circ (j_{p^*})^*\circ g)^t$. As $(\kappa_{F})^*\circ\kappa_{F^*}=\id_{F^*}$, we obtain that 
$$
f^t=((\iota_{F^*})^*\circ (j_{p^*})^*\circ g)^t\circ\kappa_{F^*}.
$$ 
Since $j_{p^*}\in\Pi_{p^*}(\iota_{F^*}(F^*),S_{p^*})$ (see \cite[Examples 2.9]{DisJarTon-95}), then $(j_{p^*})^*\in\D_p((S_{p^*})^*,(i_{F^*}(F^*))^*)$ by \cite[Theorem 2.2.2]{Coh-73}. Hence $(\iota_{F^*})^*\circ(j_{p^*})^*\circ g\in\D_p^{\H^\infty}(U,F^{**})$ with 
$$
d_p^{\H^\infty}((\iota_{F^*})^*\circ (j_{p^*})^*\circ g)
\leq d_p((\iota_{F^*})^*\circ (j_{p^*})^*)\left\|g\right\|_\infty
=\pi_{p^*}(j_{p^*}\circ\iota_{F^*})\left\|g\right\|_\infty
$$
by the ideal property of $\D_p$, Corollary \ref{messi-3} and \cite[Theorem 2.2.2]{Coh-73}. Applying Theorem \ref{messi-2} and the ideal property of $\Pi_p$, we deduce that $f^{t}=((\iota_{F^*})^*\circ (j_{p^*})^*\circ g)^t\circ\kappa_{F^*}\in\Pi_{p^*}(F^*,\H_\infty(U))$. Again, Theorem \ref{messi-2} gives that $f\in\mathcal{D}^{\H_\infty}_p(U,F)$ with $d_p^{\H^\infty}(f)=\pi_{p^*}(f^t)$. Moreover,  
\begin{align*}
d_p^{\H^\infty}(f)
&=\pi_{p^*}(((\iota_{F^*})^*\circ (j_{p^*})^*\circ g)^t\circ\kappa_{F^*})\\
&\leq \pi_{p^*}(((\iota_{F^*})^*\circ (j_{p^*})^*\circ g)^t)\left\|\kappa_{F^*}\right\|\\
&\leq d_p^{\H^\infty}((\iota_{F^*})^*\circ (j_{p^*})^*\circ g)\\
&\leq\pi_{p^*}(j_{p^*}\circ\iota_{F^*})\left\|g\right\|_\infty\\
&\leq \pi_{p^*}(j_{p^*})\left\|\iota_{F^*}\right\|\left\|g\right\|_\infty\\
&\leq \left\|g\right\|_\infty .
\end{align*}
\end{proof}



\begin{thebibliography}{99}
\bibitem{AchMez-07} D. Achour and L. Mezrag, On the Cohen strongly $p$-summing multilinear operators, J. Math. Anal. Appl. \textbf{327} (2007), no. 1, 550--563.
\bibitem{AngFer-20} J. C. Angulo-L\'opez and M. Fern\'andez-Unzueta, Lipschitz $p$-summing multilinear operators, J. Funct. Anal. \textbf{279} (2020), no. 4, 108572, 20 pp.
\bibitem{AreBatHieNeu-01} W. Arendt, C. J. K. Batty, M. Hieber and F. Neubrander, Vector-valued Laplace transforms and Cauchy problems, Monographs in Mathematics, Vol. 96, Birkhauser Verlag, Basel-Boston-Berlin, 2001. 
\bibitem{ArBoPelRu10} R. Aron, G. Botelho, D. Pellegrino and P. Rueda, Holomorphic mappings associated to composition ideals of polynomials. Atti Accad. Naz. Lincei Rend. Lincei Mat. Appl. \textbf{21} (2010), no. 3, 261--274.
\bibitem{Coh-73} J. S. Cohen, Absolutely $p$-summing, $p$-nuclear operators and their conjugates, Math. Ann. \textbf{201} (1973) 177--200.
\bibitem{DisJarTon-95} J. Diestel, H. Jarchow and A. Tonge, Absolutely summing operators, Cambridge Studies in Advanced Mathematics, vol. 43, Cambridge University Press, Cambridge, 1995. 
\bibitem{Dim-03} V. Dimant, Strongly $p$-summing multilinear mappings, J. Math. Anal. Appl. \textbf{278} (2003) 182--193.
\bibitem{FarJoh-09} J. Farmer and W. B. Johnson, Lipschitz $p$-summing operators, Proc. Amer. Math. Soc. \textbf{137} (2009) 2989--2995. 
\bibitem{Gro-55} A. Grothendieck, Produits tensoriels topologiques et espaces nucl\'{e}aires, Memoirs American Mathematical Society 16, Providence, Rhode Island 1955.
\bibitem{Jar-81} H. Jarchow, Locally convex spaces. Mathematische Leitf\"aden. B. G. Teubner, Stuttgart, 1981. 
\bibitem{JimRuiSep-22} A. Jim{\'e}nez-Vargas, D. Ruiz-Casternado and J. M. Sepulcre, On holomorphic mappings with compact type range, avalaible at https://arxiv.org/abs/2209.01576.
\bibitem{Mat-96} M. C. Matos, Absolutely summing holomorphic mappings, An. Acad. Brasil. Cienc. \textbf{68} (1996) 1--13. 
\bibitem{Muj-86} J. Mujica, Complex analysis in Banach spaces. Holomorphic functions and domains of holomorphy in finite and infinite dimensions. North-Holland Mathematics Studies, 120. Notas de Matem\'atica [Mathematical Notes], 107. North-Holland Publishing Co., Amsterdam, 1986.
\bibitem{Muj-91} J. Mujica, Linearization of bounded holomorphic mappings on Banach spaces, Trans. Amer. Math. Soc. \textbf{324} (1991), 867--887.
\bibitem{Muj-92} J. Mujica, Linearization of holomorphic mappings of bounded type. Progress in functional analysis (Pe\~{n}\'iscola, 1990), 149--162, North-Holland Math. Stud., 170, North-Holland, Amsterdam, 1992.
\bibitem{Muj-05} J. Mujica, Holomorphic functions on Banach spaces, Note Mat. \textbf{25} (2005/06), no. 2, 113--138.
\bibitem{Pel-03} D. Pellegrino, Strongly almost summing holomorphic mappings, J. Math. Anal. Appl. \textbf{287} (2003), no. 1, 244--252.
\bibitem{Pie-67} A. Pietsch, Absolut $p$-summierende Abbildungen in normierten R\"aumen, Studia Math. \textbf{28} (1967), 333--353. 
\bibitem{Pie-84} A. Pietsch, Ideals of multilinear functionals (designs of a theory), in: Proceedings of the Second International Conference on Operator Algebras, Ideals, and Their Applications in Theoretical Physics, Leipzig, 1983, in: Teubner-Texte Math., vol. 67, Teubner, Leipzig, 1984, pp. 185--199.
\bibitem{Rya-02} R. A. Ryan, Introduction to tensor products of Banach spaces, Series: Springer Monographs in Mathematics, Springer, 2002. 
\bibitem{Saa-15} K. Saadi, Some properties for Lipschitz strongly $p$-summing operators, J. Math. Anal. Appl. \textbf{423} (2015) 1410--1426.
\bibitem{YahAchRue-16} R. Yahi, D. Achour and P. Rueda, Absolutely summing Lipschitz conjugates, Mediterr. J. Math. \textbf{13} (2016), no. 4, 1949--1961.
\end{thebibliography}
\end{document}